\documentclass[twoside,11pt]{article}

\usepackage[preprint]{jmlr2e}
\usepackage{mathtools}
\usepackage{latexsym}
\usepackage{listings}
\usepackage{xcolor}
\usepackage{multirow}
\usepackage{enumitem}
\usepackage{siunitx}
\usepackage{hyperref}
\usepackage{times}
\usepackage{algorithm}
\usepackage{algpseudocode}

\newtheorem{assumption}{Assumption}
\newtheorem{thm}{Theorem}
\newtheorem{prop}{Proposition}
\newtheorem{rmk}{Remark} 
\newcommand{\mG}{\mathcal{G}}
\newcommand{\mF}{\mathcal{F}}
\newcommand{\ve}{\varepsilon}
\newcommand{\ep}{\epsilon}
\newcommand{\sg}{\sigma}

\newcommand{\RR}{\mathbb{R}}

\newcommand{\mO}{\mathcal{O}}
\newcommand{\mL}{\mathcal{L}}
\newcommand{\mT}{\mathcal{T}}

\newcommand{\PP}{\mathbb{P}}
\newcommand{\EE}{\mathbb{E}}
\newcommand{\tp}{^{\top}}
\newcommand{\ptp}{^{\prime \top}}
\newcommand{\rd}{\,\mathrm{d}}

\DeclareMathOperator{\Tr}{Tr}

\jmlrheading{1}{2022}{1-48}{4/00}{10/00}{meila00a}{Mo Zhou and Jianfeng Lu}

\ShortHeadings{Single Timescale Actor-Critic LQR}{Zhou and Lu}
\firstpageno{1}

\begin{document}

\title{Single Timescale Actor-Critic Method to Solve the Linear Quadratic Regulator with Convergence Guarantees}

\author{\name Mo Zhou \email mo.zhou366@duke.edu \\
       \addr Department of Mathematics\\
       Duke University\\
       Durham, NC 27708, USA
       \AND
       \name Jianfeng Lu \email jianfeng@math.duke.edu \\
       \addr Department of Mathematics, Department of Physics, and Department of Chemistry\\
       Duke University\\
       Durham, NC 27708, USA}
\editor{Kevin Murphy and Bernhard Sch{\"o}lkopf}

\maketitle

\begin{abstract}
We propose a single timescale actor-critic algorithm to solve the linear quadratic regulator (LQR) problem. A least squares temporal difference (LSTD) method is applied to the critic and a natural policy gradient method is used for the actor. We give a proof of convergence with sample complexity $\mO(\ve^{-1} \log(\ve^{-1})^2)$. The method in the proof is applicable to general single timescale bilevel optimization problems. We also numerically validate our theoretical results on the convergence.
\end{abstract}


\section{Introduction}
Reinforcement learning (RL) is a semi-supervised learning model that learns to take actions and interact with the environment in order to maximize the expected reward \citep{sutton2018reinforcement}. It has a wide range of applications, including robotics \citep{kober2013reinforcement}, traditional games \citep{silver2016mastering}, and traffic light control \citep{wiering2000multi}. RL is closely related to the optimal control problem \citep{bertsekas2019reinforcement}, where one usually minimizes the expected cost instead of maximizing the reward. Among all the control problems, the LQR \citep{anderson2007optimal} is the cleanest setup to analyze theoretically and has many applications \citep{hashim2019optimal, ebrahim2010application}.
Many research has been devoted to LQR. Early research mostly focused on model-based methods, such as deriving the explicit solution of the LQR with known dynamics. This research showed that the optimal control is a linear function of the state and the coefficient can be obtained by solving the Riccati equation \citep{anderson2007optimal}. Recent research focuses more on the model-free setting in the context of RL, where the algorithm does not know the dynamic and has only observations of states and rewards \citep{tu2018least, mohammadi2021convergence}.

The actor-critic method \citep{konda2000actor} is a class of algorithms that solve the RL or optimal control problems through alternately updating the actor and the critic. In this framework, we solve for both the control and the value function, which is the expected cost w.r.t. the initial state (and action). The control is known as the actor, so in the actor update, we improve the control in order to minimize the cost; i.e., policy improvement.  The value function is known as the critic. Hence, in the critic update, we evaluate a fixed control through computing the value function; i.e., policy evaluation. 

On a broader scale, the actor-critic method belongs to the bilevel optimization problem \citep{sinha2017review, bard2013practical}, as it is an optimization problem (higher-level problem) whose constraint is another optimization problem (lower-level problem). In the actor-critic method, the higher-level problem is to minimize the cost (the actor) and the lower-level problem is to let the critic be equal to value function corresponding to the control, which is equivalent to minimizing the expected squared Bellman residual \citep{bradtke1996linear}. The major difficulty of a bilevel optimization problem is that when the lower-level problem is not solved exactly, the error could propagate to the higher-level problem and accumulate in the algorithm. One approach to overcome this problem is the two timescale method \citep{konda2000actor, wu2020finite, zeng2021two}, where the update of lower-level problem is in a time scale that is much faster than the higher-level one. This method suffers from high computational costs because of the lower-level optimization. Another method is to modify the update direction to improve accuracy \citep{kakade2001natural}, which also introduces extra cost. In order to reduce the cost, we seek an efficient single timescale method to solve LQR.

\subsection{Our contributions}
In this paper, we consider a single timescale actor-critic algorithm to solve the LQR problem. We apply an LSTD method \citep{bradtke1996linear} for the critic and a natural policy gradient method \citep{kakade2001natural} for the actor. For the critic, we derive an explicit expression for the gradient and design a sample method with the desired accuracy. For the actor, we apply a natural policy gradient method borrowed from \cite{fazel2018global}. We give a proof of convergence with sample complexity $\mO(\ve^{-1} \log(\ve^{-1})^2)$ to achieve an $\ve$-optimal solution. To the best of our knowledge, our work is the first single timescale actor-critic method to solve the LQR problem with provable guarantees.

Our work not only solves the specific LQR problem but also advances the study of convergence for single timescale bilevel optimization. In our proof of convergence, we construct a Lyapunov function that involves both the critic error and the actor loss. We show that there is a contraction of the Lyapunov function in the algorithm. If we consider the actor and the critic separately, the critic error becomes an issue when we want to show an improvement of the actor and vice versa. Therefore, the higher and lower level problems have to be analyzed simultaneously for a single timescale algorithm. 

\subsection{Related works}
Let us compare our work with related ones in the literature.
Perhaps the most closely related work to ours is by \cite{fu2020single}. They consider a single timescale actor-critic method to solve the optimal control problem with discrete state and action spaces, while we solve the LQR problem with continuous state and action spaces. They add an entropy regularization in the loss function and achieve a sample complexity of $\mO(\ve^{-2})$ with linear parameterization. 

For two timescale approaches, \cite{yang2019global} study a two timescale actor-critic algorithm to solve the LQR problem in continuous space. They also use a natural policy gradient method for the actor \citep{fazel2018global}. For the critic, they reformulate policy evaluation into a minimax optimization problem using Fenchel's duality. Several critic steps are performed between two actor steps and their final sample complexity is $\mO(\ve^{-5})$. \cite{zeng2021two} study a bilevel optimization problem that is applied to a two timescale actor-critic algorithm on LQR. They obtain a complexity of $\mO(\ve^{-3/2})$. They have assumed strong convexity of the higher-level loss function (actor) while our analysis does not require such assumptions.

Besides model-free approaches, another way to solve the LQR problem is to first learn the model through the system identification approach and then solve the model-based LQR. For example, \cite{dean2020sample} use a least square system identification approach to learn the model parameter and then solve the LQR, with sample complexity $\mO(\ve^{-2})$.

As can be seen from the above discussions, our single timescale algorithm achieves a lower sample complexity $\mO(\ve^{-1} \log(\ve^{-1})^2)$, which is an improvement over previously proposed algorithms.

For the general bilevel optimization problem, we refer the reader to \cite{chen2022single}, where the authors summarize the existing bilevel algorithms and propose a STABLE method with $\mO(\ve^{-1})$ sample complexity under strong convexity assumption.

\smallskip

The rest of this paper is organized as follows. In Section \ref{sec:background}, we introduce the theoretical background of the LQR problem. In Section \ref{sec:algorithm}, we describe the algorithm for the LQR problem and our choice of parameters. In Section \ref{sec:proof}, we give the outline of the convergence proof of the algorithm, with proof details in the appendix. The numerical examples are also deferred to the appendix. 

\section{Theoretical background}\label{sec:background}
First, we clarify some notations.
We use $\|\cdot\|$ to denote the operator norm of a matrix and $\|\cdot\|_F$ to denote the Frobenius norm of a matrix. When we write $M \ge c$ where $M$ is a symmetric matrix and $c$ is a number, we mean $M - c I$ is positive semi-definite. Similarly, $M>c$ means $M - c I$ is positive definite.

We consider a discrete-time Markov process $\{x_s\}$ on a filtered probability space $(\Omega, \mF, \{\mF_s\}, \PP)$:
\begin{equation*}
    x_{s+1} = A x_s + B u_s + \xi_s,
\end{equation*}
where $x_s \in \RR^d$ is an adapted state process, $u_s \in \RR^k$ is the adapted control process, $A \in \RR^{d \times d}$ and $B \in \RR^{d \times k}$ are two fixed matrices. $\xi_s \sim N(0, D_{\xi})$ is independent noise.
The initial state $x_0 \sim \rho_0$, with some initial distribution $\rho_0$. 

The goal is to minimize the infinite horizon cost functional
\begin{equation}\label{eq:cost}
    J(\{u_s\}) = \lim_{S \to \infty} \EE \left[ \dfrac{1}{S} \sum_{s=0}^{S-1} c(x_s, u_s) \right],
\end{equation}
where $c(x, u) = x\tp Q x + u\tp R u$ is the one-step cost, with $Q \in \RR^{d \times d}$ and $R \in \RR^{k \times k}$ being positive definite. Theoretical results guarantee that the optimal control $u^*$ is linear in $x$: $u^*_s = - K^* x_s$.
If the model is known, we can obtain the optimal control parameter by $K^* = (R + B\tp P^*B)^{-1} B\tp P^* A$ where $P^*$ is the solution to the Riccati equation  \citep{anderson2007optimal}
\begin{equation}\label{eq:Riccati}
    P^* = Q + A\tp P^* A - A\tp P^* B (R + B\tp P^* B)^{-1} B\tp P^* A.
\end{equation}
In this work, we consider the model-free setting (i.e., the algorithm does not have access to $A$, $B$, $D_{\xi}$, $Q$, $R$). We will use a stochastic policy parametrized as 
\begin{equation}\label{eq:policy}
    u_s \sim \pi_K := N(-K x_s, \sigma^2 I_k)
\end{equation}
to encourage exploration, where $\sigma > 0$ is a fixed constant. Here, we use $\pi_K$ to denote the distribution while we will not distinguish in notation a probability distribution with its density. We remark that adding exploration does not change the optimal $K^*$ because the optimal policy parameters with or without exploration satisfy the same Riccati equation while adding exploration would help the convergence of the algorithm. Under this policy, the cost functional \eqref{eq:cost} is also denoted by $J(K)$  and the state trajectory can be rewritten as
\begin{equation*}
    x_{s+1} = A x_s + B (-K x_s + \sigma \omega_s) + \xi_s =: (A -BK) x_s + \ep_s
\end{equation*}
where $\omega_s \sim N(0, I_k)$ and $\ep_s \sim N(0, D_{\ep})$ with $D_{\ep} = D_{\xi} + \sg^2 BB\tp$ being positive definite. Let $\rho(\cdot)$ denote the spectral radius of a matrix. When $\rho(A-BK) < 1$, the state process has a stationary distribution $N(0, D_K)$, where $D_K \in \RR^{d \times d}$ satisfies the Lyapunov equation
\begin{equation}\label{eq:DK}
    D_K = D_{\ep} + (A-BK) D_K (A-BK)\tp.
\end{equation}
In order to understand \eqref{eq:DK}, let us assume that $x \sim N(0, D_K)$ follows the stationary distribution. Then, $x' = (A -BK) x + \ep \sim N(0, (A-BK) D_K (A-BK)\tp + D_{\ep})$  also follows the stationary distribution, which leads to \eqref{eq:DK}. $D_K$ can also be expressed in terms of a series: since $\rho(A-BK) < 1$, we can recursively plug in the definition of $D_K$ into the right hand side of \eqref{eq:DK} and obtain
\begin{equation}\label{eq:DK_series}
    D_K = \sum_{s=0}^{\infty} (A-BK)^s D_{\ep} ((A-BK)\tp)^s.
\end{equation}
From here on, the notation $\EE_{K}$ means the expectation with $x$ (or $x_0$) $\sim N(0, D_K)$ if not specified and $u$ (or $u_s$) $\sim \pi_K$.
The state-action value function (Q function) and the state value function with respect to a control $\{u_s\}$ are defined by 
\begin{equation}\label{eq:value}
\begin{aligned}
    Q(x, u) &= \sum_{s=0}^{\infty} \left( \EE \left[ c(x_s, u_s) ~|~ x_0=x, u_0=u \right] - J(\{u_s\}) \right)\\
    V(x) &= \sum_{s=0}^{\infty} \left( \EE \left[ c(x_s, u_s) ~|~ x_0=x \right] - J(\{u_s\}) \right) = \EE_{u} \left[ Q(x,u) \right]
\end{aligned}
\end{equation}
respectively. $V(x)$ is the expected extra cost if we start at $x_0=x$ and follow a given policy. $Q(x,u)$ is the expected extra cost if we start at $x_0=x$, take the first action $u_0=u$, and then follow a given policy. These two functions are crucial in reinforcement learning. If the policy $\pi_K$ follows \eqref{eq:policy}, then the two functions in \eqref{eq:value} are denoted by $Q_K(x,u)$ and $V_K(x)$ respectively. By definition, for any $x$ and $u$, it satisfies the Bellman equation:
\begin{equation}\label{eq:Bellman_eqn}
    Q_K(x,u) = c(x,u) - J(K) + \EE_K\bigl[ Q_K(x',u') \mid x, u \bigr],
\end{equation}
where $(x', u')$ is the next state-action pair starting from $(x, u)$. 

We define $P_K$ as the solution to the following matrix valued equation 
\begin{equation}\label{eq:PK}
    P_K = (Q + K\tp R K) + (A-BK)\tp P_K (A-BK).
\end{equation}
$P_K$ can be interpreted as the second order adjoint state, and $P_K x_t$ is the shadow price for the system (see for example \cite{yong1999stochastic}).
We have the following two properties to illustrate the importance of $P_K$. The proofs are deferred to the appendix.
\begin{prop}\label{prop:cost_expression}
Let the policy $\pi_K$ be defined by \eqref{eq:policy} with $\rho(A - BK) < 1$. Then the cost function and its gradient w.r.t. $K$ have the following explicit expressions:
\begin{align}\label{eq:JK}
    J(K) & = \Tr(D_{\ep} P_K) + \sigma^2 \Tr(R), \\
\label{eq:grad_JK}
    \nabla_K J(K) & = 2 \left[ (R + B\tp P_K B)K - B\tp P_K A \right] D_K.
\end{align}
\end{prop}
\begin{rmk}
In the LQR problem, we usually assume that $D_K$ is positive definite and hence invertible. Therefore, the critical point for $J(K)$ (i.e., when $\nabla_K J(K)=0$) satisfies $K = (R + B\tp P_KB)^{-1} B\tp P_K A$. If we substitute this into \eqref{eq:PK}, we recover the Riccati equation \eqref{eq:Riccati}.
\end{rmk}

\begin{prop}\label{prop:value_expression}
Let the policy $\pi_K$ be defined by \eqref{eq:policy} with $\rho(A - BK) < 1$. Then the value functions have the following explicit expressions:
\begin{equation*}
    V_K(x) = x\tp P_K x - \Tr(D_K P_K),
\end{equation*}
\begin{multline}\label{eq:Q_fcn}
    Q_K(x,u) = \begin{bmatrix}  x\tp & u\tp \end{bmatrix} \begin{bmatrix} Q + A\tp P_K A & A\tp P_K B \\ B\tp P_K A & R + B\tp P_K B \end{bmatrix} \begin{bmatrix}  x \\ u \end{bmatrix} \\ - \sigma^2 \Tr(R + P_K BB\tp) - \Tr(D_K P_K).
\end{multline}
\end{prop}
If we concatenate $x$ and $u$ in the dynamic equation, the process can be written as 
\begin{equation*}
    \begin{bmatrix}  x_{s+1} \\ u_{s+1} \end{bmatrix} = \begin{bmatrix}  A&B \\ -KA & -KB \end{bmatrix} \begin{bmatrix}  x_s \\ u_s \end{bmatrix} + \begin{bmatrix}  \xi_s \\ -K \xi_s + \sigma \omega_s \end{bmatrix}.
\end{equation*}
We simplify the expression by introducing some new notations: $z_s = [x_s\tp, u_s\tp]\tp$, thus $z_{s+1} = E z_s + \widetilde{\ep}_s$,
where
\begin{equation}\label{eq:def_concat}
    E = \begin{bmatrix}  A&B \\ -KA & -KB \end{bmatrix},~ \text{and} ~~ \widetilde{\ep}_s \sim N(0, \Sigma_{\ep}) := N\left(0,  \begin{bmatrix}  D_{\xi} & -D_{\xi} K\tp \\ -K D_{\xi} & K D_{\xi} K\tp + \sigma^2 I_k \end{bmatrix} \right).
\end{equation}
The ergodicity of the dynamics is guaranteed if $\rho(A - BK) = \rho(E) < 1$, where the identity $\rho(A - BK) = \rho(E)$ can be verified from 
$$\rho(E) = \rho\left( \begin{bmatrix}  I_d \\ -K \end{bmatrix} \begin{bmatrix}  A & B \end{bmatrix} \right) = \rho\left( \begin{bmatrix}  A & B \end{bmatrix} \begin{bmatrix}  I_d \\ -K \end{bmatrix}  \right) = \rho(A - BK).$$
The stationary distribution for $z$ is given by  
\begin{equation}\label{eq:stationary_z}
    z \sim N(0, \Sigma_K) := N\left( 0, \begin{bmatrix}  D_{K} & - D_{K} K\tp \\ - K D_{K} & K  D_{K} K\tp + \sigma^2 I_k \end{bmatrix} \right) 
\end{equation}
and we have $\Sigma_K = \Sigma_{\ep} + E \Sigma_K E\tp$.

\section{The actor-critic algorithm}\label{sec:algorithm}
In this section, we present our specific design of the algorithm under the actor-critic framework. We apply an LSTD method for the policy evaluation (critic), with a detailed description for sampling the gradient of the loss function. We also use a natural policy gradient method for the policy improvement (actor). We will use $\mG_t$ to denote the filtration generated by the training process.  We use $\mO(a)$ to denote a quantity that is is bounded by a constant times $a$, where this constant only depends on the problem setting ($A$, $B$, $D_{\ep}$, $Q$, $R$, $\sigma$) and does not depend on the target accuracy or training trajectory. The dependence of the constants on the dimensions is explained in the proof of our theorem.

\subsection{Policy evaluation for the critic} \label{sec:critic_alg}
In this subsection, we describe the policy evaluation algorithm for a fixed policy $\pi_K$.
We parametrize the state-action value function by $Q^{\theta}_K$ with $\theta$ as a parameter and subscript $K$ indicating that it depends on the given policy $\pi_K$. We define the Bellman residual w.r.t. the critic parameter $\theta$ as
\begin{equation*}
    \textsf{BR}_{\theta}(x, u) = c(x,u) - J(K) + \EE_K\left[ Q^{\theta}_K(x',u') | x, u \right] - Q^{\theta}_K(x,u).
\end{equation*}
Recall the exact $Q$ function is given by \eqref{eq:Q_fcn}, accordingly, we define a feature matrix
\begin{equation}\label{eq:feature}
    \phi(x, u) = \begin{bmatrix}  x \\ u \end{bmatrix} \begin{bmatrix}  x\tp & u\tp \end{bmatrix} \in \RR^{(d+k) \times (d+k)}
\end{equation}
and parametrize the $Q$ function as
\begin{equation}\label{eq:Q_para}
    Q^{\theta}_K (x, u) = \Tr(\phi(x, u) \theta) - \theta',
\end{equation}
where $\theta \in \RR^{(d+k) \times (d+k)}$ and $\theta' \in \RR$. Here, we denote 
\begin{equation}\label{eq:theta_K}
    \theta = \begin{bmatrix}  \theta^{11} & \theta^{12} \\ \theta^{21} & \theta^{22} \end{bmatrix}, ~~ \text{which intends to approximate} ~~ \theta_K = \begin{bmatrix} Q + A\tp P_K A & A\tp P_K B \\ B\tp P_K A & R + B\tp P_K B \end{bmatrix}.
\end{equation}
The scalar parameter $\theta'$ is to approximate $\sigma^2 \Tr(R + P_K BB\tp) + \Tr(D_K P_K)$.
Recall the Bellman equation \eqref{eq:Bellman_eqn}, with parametrization \eqref{eq:Q_para}, the Bellman residual is written as 
\begin{multline*}
    \text{BR}_{\theta}(x, u) = c(x,u) - J(K) +  \langle \EE_K \left[ \phi(x', u') | x, u \right] - \phi(x,u), \theta \rangle \\
    =: c(x,u) - J(K) +  \langle \psi(x,u), \theta \rangle,
\end{multline*}
where $\langle \cdot , \cdot \rangle$ is the trace inner product and we have defined $\psi(x, u) := \EE_K\left[ \phi(x', u') | x, u \right] - \phi(x,u)$ for convenience. It is clear by definition that $\EE_K[\psi(x,u)] = 0$ (recall that $x$ follows the stationary distribution $N(0, D_K)$). The loss function for critic is then defined as the expectation of squared Bellman residual:
\begin{equation}\label{eq:critic_loss}
    L_K(\theta) = \frac12 \EE_K \left[\text{BR}_{\theta}(x, u)^2 \right] = \frac12 \EE_K \left[ \left( c(x,u) - J(K) +  \langle \psi(x,u), \theta \rangle \right)^2 \right].
\end{equation}
We will find that $\theta'$ does not affect the training, so only $\theta$ will be considered as the critic parameter from now on. According to the Bellman equation \eqref{eq:Bellman_eqn}, the unique minimizer of \eqref{eq:critic_loss} is the true parameter for the $Q$ function w.r.t. $\pi_K$. By direct computation, the gradient (as a matrix) and Hessian (as a tensor) of the loss function w.r.t. $\theta$ are
\begin{equation}\label{eq:critic_grad}
\begin{aligned}
    \nabla L_K(\theta) &= \EE_K \left[ \left( c(x,u) - J(K) +  \langle \psi(x,u), \theta \rangle \right) \psi(x, u) \right] \\
    &= \EE_K \left[ \left( c(x,u) +  \langle \psi(x,u), \theta \rangle \right) \psi(x, u) \right]
    \end{aligned}
\end{equation}
and
\begin{equation*}
    \nabla^2 L_K(\theta) =  \EE_K \left[ \psi(x, u) \otimes \psi(x, u) \right],
\end{equation*}
where $\otimes$ denotes the tensor product. The loss function $L_K$ is strongly convex in $\theta$, as will be shown later.

To minimize the loss \eqref{eq:critic_loss}, we use stochastic gradient descent method. Thus, we need an accurate sample estimate of $\nabla L_K(\theta)$ for given $K$ and $\theta$.  For simplicity of notation, we denote
\begin{equation}\label{eq:f}
    f(x,u) := \left( c(x,u) +  \langle \psi(x,u), \theta \rangle \right) \psi(x, u) = c(x,u) \psi(x, u) + (\psi(x, u) \otimes \psi(x, u)) \cdot \theta
\end{equation}
so that $\nabla L_K(\theta) = \EE_K[f(x,u)]$. Note that $f(x,u)$ depends on $\theta$ and $K$, while we suppress that in the notation. We decompose the sampling into two steps: the first step is to obtain $x, u$ that approximately follows the stationary distribution $N(0, \Sigma_K)$ and the second one is to sample $f(x,u)$ for given $x, u$.

For the first step, we use the Markov chain Monte Carlo (MCMC) method \citep{gilks1995markov}. Let $N_0$ and $N$ be two integers that will be determined according to the error tolerance. Starting at $x_0=0$, we sample $N$ independent trajectories of length $N_0+1$ according to the policy $\pi_K$. So, we obtain $N$ samples $\{(x^{(i)}_{N_0}, u^{(i)}_{N_0})\}_{i=1}^{N}$ that follow the distribution of $(x_{N_0}, u_{N_0})$. For each pair $(x^{(i)}_{N_0}, u^{(i)}_{N_0})$, we generate $N_1$ unbiased sample for $\psi(x^{(i)}_{N_0}, u^{(i)}_{N_0})$, given by
\begin{equation*}
    \widehat{\psi}^{(i)}_j = \phi(x^{(i,j)}, u^{(i,j)}) - \phi(x^{(i)}_{N_0}, u^{(i)}_{N_0}) ~~~~~~ j = 1, 2, \cdots, N_1
\end{equation*}
where $x^{(i,j)}, u^{(i,j)}$ are sampled independently and follow the next step distribution conditioned on $(x^{(i)}_{N_0}, u^{(i)}_{N_0})$. Here, $N_1 = \mO(1)$ is another predefined hyperparameter. We denote the mean by $\Bar{\psi}^{(i)} = \frac{1}{N_1} \sum_{j=1}^{N_1} \widehat{\psi}^{(i)}_j$. Therefore, we can obtain an unbiased sample for $f(x^{(i)}_{N_0}, u^{(i)}_{N_0})$ by
\begin{equation}\label{eq:f_hat}
\begin{aligned}
     \widehat{f}(x^{(i)}_{N_0}, u^{(i)}_{N_0}) =&  \frac{1}{N_1} \sum_{j=1}^{N_1} c(x^{(i)}_{N_0}, u^{(i)}_{N_0}) \widehat{\psi}^{(i)}_j\\
     & + \biggl[ \frac{1}{N_1} \sum_{j=1}^{N_1} \widehat{\psi}^{(i)}_j \otimes \widehat{\psi}^{(i)}_j - \frac{1}{N_1-1} \sum_{j=1}^{N_1} (\widehat{\psi}^{(i)}_j - \Bar{\psi}^{(i)}) \otimes (\widehat{\psi}^{(i)}_j - \Bar{\psi}^{(i)}) \biggr] \cdot \theta.
\end{aligned}
\end{equation}
Note that the first and second terms in the square bracket are unbiased samples for $\EE[\widehat{\psi}^{(i)}_j \otimes \widehat{\psi}^{(i)}_j]$ and $\text{Cov}(\widehat{\psi}^{(i)}_j)$ respectively, which implies that the square bracket is an unbiased sample for $\psi(x^{(i)}_{N_0}, u^{(i)}_{N_0}) \otimes \psi(x^{(i)}_{N_0}, u^{(i)}_{N_0})$. Finally, the sample of gradient $\nabla L_K(\theta)$ is given by
\begin{equation}\label{eq:gradient_sample}
    \widehat{\nabla L}_{K}(\theta) = \frac1N \sum_{i=1}^N \widehat{f}(x^{(i)}_{N_0}, u^{(i)}_{N_0}).
\end{equation}
The one-step sample complexity is $\mO(N_0 N_1 N)$.
We remark that our LSTD is similar to a $\mathrm{TD}(0)$ algorithm, except that we have $N$ trajectories and we omit $J(K)$ in \eqref{eq:critic_grad}.
Denote $L_{K_t}(\theta)$ by $L_t(\theta)$ for simplicity. We also denote $\theta_t$ the critic parameter at step $t$.
The gradient sample at step $t$ (in matrix form) is denoted by $\widehat{\nabla L_t}(\theta_t)$ and the critic update is given by
\begin{equation*}
    \theta_{t+1} = \theta_t - \alpha_t \widehat{\nabla L_t}(\theta_t),
\end{equation*}
where $\alpha_t$ is the step size for the critic.

\subsection{Policy improvement for the actor}

For the actor algorithm, we borrow the idea from \cite{fazel2018global} which considered a policy gradient algorithm for the LQR problem. A similar approach is also studied by \cite{yang2019global, zeng2021two}.

Motivated by the form of the gradient \eqref{eq:grad_JK}, we define
\begin{equation}\label{eq:GK}
    G_K := (R + B\tp P_K B)K - B\tp P_K A, 
\end{equation}
so that $\nabla_K J(K) = 2 G_K D_K$.
Therefore, a vanilla policy gradient algorithm looks like
\begin{equation*}
    K_{t+1} = K_t - \beta_t G_{K_t} D_{K_t},
\end{equation*}
where $G_{K_t}$ and $D_{K_t}$ may be replaced by some estimates and $\beta_t$ is the step size for the actor. 

Instead of the vanilla policy gradient, we would consider the commonly used variant known as the natural policy gradient method \citep{kakade2001natural}. The natural policy gradient uses the inverse Fisher information matrix to precondition the gradient so that the gradient is taken w.r.t. the metric induced by the Hessian of the loss function \citep{peters2008natural}. This method has been studied in e.g., \citep{kakade2001natural, peters2008natural, bhatnagar2009natural, liu2020improved}. 
The Fisher information matrix at each state $x$ is given by 
\begin{equation}\label{eq:Fisher_state}
    F_x(K) = \EE_{u\sim \pi_K} \left[ \nabla_K \log(\pi_{K}(u | x)) \otimes  \nabla_K \log(\pi_{K}(u | x)) \right],
\end{equation}
which is a tensor in $\RR^{k \times d} \otimes \RR^{k \times d}$ as $K \in \RR^{k \times d}$ is a matrix. Then, the (average) Fisher information matrix is defined as
\begin{equation*}
    F(K) = \EE_{x \sim N(0, D_K)} \left[ F_x(K) \right] = \EE_{K} \left[ \nabla_K \log(\pi_{K}(u | x)) \otimes  \nabla_K \log(\pi_{K}(u | x)) \right].
\end{equation*}
Under the metric induced by the Hessian, the steepest descent direction of $J(K)$ is given by 
\begin{equation*}
    - \widetilde{\nabla} J(K) = - F(K)^{-1} \cdot \nabla_K J(K) = - 2 F(K)^{-1} \cdot G_K D_K,
\end{equation*}
where for $F(K)^{-1}$, we view the tensor $F(K)$ as a linear operator $\RR^{k \times d} \to \RR^{k \times d}$, so $F(K)^{-1}$ is the inverse operator. 
The following property gives a simple expression of $\widetilde{\nabla} J(K)$. The proof is in the appendix.
\begin{prop}\label{prop:rescaled_gradient} We have 
\begin{equation}\label{eq:rescaled_gradient}
    \widetilde{\nabla} J(K) = 2 \sigma^2 G_K.
\end{equation}
\end{prop}
Recall that $G_K = (R + B\tp P_K B)K - B\tp P_K A$. Hence, $G_K = \theta^{22}_K K - \theta^{21}_K$ where $\theta_K$ is the true parameter w.r.t. policy $\pi_K$, given by \eqref{eq:theta_K}.
Therefore, the actor update is given by 
\begin{equation}\label{eq:actor_update}
    K_{t+1} = K_t - \beta_t (\theta^{22}_{t} K_t - \theta^{21}_{t}) =: K_t - \beta_t \widehat{G}_{K_t},
\end{equation}
where the constant $2\sigma^2$ is absorbed in the step size $\beta_t$ and we have defined $\widehat{G}_{K_t} := \theta^{22}_{t} K_t - \theta^{21}_{t}$. Recall that we use $\mG_t$ to denote the filtration generated by the training process. Since $K_{t+1}$ is deterministic in $\theta_t$ and $K_t$, $K_{t+1}$ is $\mG_t$-measurable. 

\subsection{Assumptions and main result}
Here we state some technical assumptions for our result.
\begin{assumption}\label{assump}
We assume that
\begin{enumerate}
    \item There exists a constant $\rho \in (0,1)$ such that $\rho(A-B K_t) = \rho(E_t) \le \rho$, for all $t$.
    \item There exist constants $c_A, c_E, c_{\theta}, c_K > 0$ such that $\|A-B K_t\| \le c_A$, $\|E_t\| \le c_E$, $\|\theta_t\|_F \le c_{\theta}$, and $\|K^*\|, \|K_t\| \le c_K$ for all $t$.
    \item $D_{\ep}$ is positive definite with minimum eigenvalue $\sigma_{min}(D_{\ep}) > 0$.
\end{enumerate}
\end{assumption}

\begin{rmk}
In the assumption, $E_t$ is defined by \eqref{eq:def_concat} with $K$ replaced by $K_t$. The first assumption is common in the analysis of the LQR problem \citep{fazel2018global, yang2019global}. A theoretical guarantee for this condition is hard to obtain, while we will present some numerical examples to support this assumption.
The second assumption gives upper bounds for several matrices, which is made to avoid technical tedious works to control the probability of the random trajectory hitting unfavorable regions. One potential way to alleviate this assumption is to define a projection map that reduces the size of $\theta_t$ or $K_t$ whenever it is out of range \citep{konda2000actor, bhatnagar2009natural}, which is left for future work.
The third assumption is necessary to make the problem non-degenerate (cf. Lemma \ref{lem:actor_improvement} below). 
\end{rmk}

Next, we specify the choice of parameters in the algorithm. We  initialize $\theta_0=0$, $K_0 = 0$ for simplicity. Fixing the error tolerance $\ve>0$, we set the step sizes $\alpha_t$ and $\beta_t$ to be constant in $t$:
\begin{equation}\label{eq:stepsize}
    \alpha_t = \dfrac{\sigma_{min}(D_{\ep})}{16 c_L^2 c_3 \kappa}\ve ~~~~~~ \beta_t = \dfrac{\sigma_{min}(D_{\ep})}{16 c_L^2 c_3 \kappa^2}\ve
\end{equation}
where
\begin{equation}\label{eq:kappa}
    \kappa = \max\left(\dfrac{3 \sigma_{min}(D_{\ep})}{2 c_3 \mu_{\sigma}}, \dfrac{4c_1^2}{\mu_{\sigma} \sigma_{min}(D_{\ep})}, \dfrac{3 c_D c_K^2}{\mu_{\sigma}}\right).
\end{equation}
Here, every parameter appearing in \eqref{eq:stepsize} and \eqref{eq:kappa}, except $\alpha_t$, $\beta_t$, or $\ve$, are constants of order $\mO(1)$:
\begin{enumerate}[noitemsep,topsep=0pt,parsep=0pt,partopsep=0pt]
    \item $c_L^2$ is the upper bound for $\EE [ \| \widehat{\nabla L_t}(\theta_t) \|^2_F ~|~ \mG_t ]$ that is in Lemma \ref{lem:critic_gradient_accuracy}; 
    \item $c_3$ illustrates the geometry of $J(K)$, with details in Lemma \ref{lem:gradient_dominant}; 
    \item In Lemma \ref{lem:critic_convex}, we will show that the critic loss is $\mu_{\sigma}$-strongly convex;
    \item $c_1$ is a Lipschitz constant for $\theta_K$ w.r.t. $K$ that is specified in Lemma \ref{lem:thetaK};
    \item $c_D$ is an upper bound for $\|D_{K_t}\|$ and $\|D_{K^*}\|$ that is specified in Lemma \ref{lem:bound}.
\end{enumerate}
\smallskip 
It is easy to verify that the step sizes satisfies the following inequalities:
\begin{equation}\label{eq:stepsizes}
    \dfrac{\sigma_{min}(D_{\ep})}{c_3} \beta_t \le \frac23 \mu_{\sigma} \alpha_t, ~ 
    \dfrac{\sigma_{min}(D_{\ep})}{\beta_t} \ge (\dfrac{3}{\alpha_t \mu_{\sigma}} + 2) c_1^2 + (\|R\| + c_P \|B\|^2), ~
    \text{and} ~
    \frac13 \alpha_t \mu_{\sigma} \ge \beta_t c_D c_K^2,
\end{equation}
where we need to assume that $\ve$ is small enough such that $1/ (\mu_{\sigma} \alpha_t) \ge 2 + (\|R\| + c_P \|B\|^2)/c_1^2$ for the second inequality. The total number of iterations is $T = \mO(\frac{1}{\ve} \log(\frac{1}{\ve}))$ such that
$$(1 - \beta_t c_4)^T L_0 < \ve,$$
where $L_0 = \mO(1)$ is the initial Lyapunov function that is specified at the beginning of the proof for Theorem \ref{thm:main1} and $c_4=\mO(1)$ is a positive constant that is also specified in the proof for Theorem \ref{thm:main1}. The number of samples $N$, the length of trajectory $N_0$ each step, and the sub-sample size $N_1$, are set to be $N = \mO(1)$, $N_0 = \mO(\log(\frac{1}{\ve}))$, and $N_1=\mO(1)$, in order to achieve desired accuracy for the sample of critic gradient, with details in Lemma \ref{lem:critic_gradient_accuracy}. 
Here, $\frac{\alpha_t}{\beta_t} = \kappa = \mO(1)$ implies that our algorithm has single timescale. In such algorithm, the actor and the critic are interdependent, which makes the analysis challenging.
We summarize the actor-critic algorithm in Algorithm \ref{alg:lqr}.

\begin{algorithm}
\caption{Single timescale actor-critic algorithm for LQR}\label{alg:lqr}
\begin{algorithmic}
\Require Training steps T, step sizes $\alpha_t$, $\beta_t$, sample size $N$, $N_0$, and $N_1$
\Ensure critic parameter $\theta_T$, actor parameter $K_T$
\State initialization: critic parameter $\theta_{0} = 0$ and actor parameter $K_0 = 0$
\For{$t=0$ \textbf{to} $T-1$}
\State Sample $\widehat{\nabla L_t}(\theta_t)$ according to \eqref{eq:gradient_sample} \Comment{critic steps}
\State $\theta_{t+1} = \theta_{t} - \alpha_t \widehat{\nabla L_t}(\theta_t)$
\medskip
\State $K_{t+1} = K_t - \beta_t (\theta^{22}_{t} K_t - \theta^{21}_{t}) $ \Comment{actor steps}
\EndFor
\end{algorithmic}
\end{algorithm}

The main result of our work is the following convergence theorem. 
\begin{thm}[Main theorem]\label{thm:main1}
Under Assumption \ref{assump}, for any $\ve > 0$ that is sufficiently small, Algorithm~\ref{alg:lqr}, with the choice of parameters discussed above, has sample complexity $\mO(\frac{1}{\ve} \log(\frac{1}{\ve})^2 )$. Moreover, the terminal error satisfies
\begin{equation*}
    \EE [\|\theta_T - \theta_{K_T}\|_F^2] \le \ve ~~~ \text{and} ~~~ \EE [J(K_T) - J(K^*)] \le \ve.
\end{equation*}
\end{thm}

\begin{rmk}
The number of steps is $T = \mO(\frac{1}{\ve} \log(\frac{1}{\ve}))$ and the one-step complexity is $\mO(\log(\frac{1}{\ve}))$. Therefore, the total complexity is $\mO(\frac{1}{\ve}~ \log(\frac{1}{\ve})^2)$. This theorem tells us that we have small error for both the critic and the actor. If we want error estimate for $\|K_T - K^*\|_F$ or $\|\theta_T - \theta^*\|_F$, we will need extra assumption such as strong convexity of $J(K)$ in $K$.
\end{rmk}

We believe the complexity $\mO(\frac{1}{\ve} \log(\frac{1}{\ve})^2)$ is nearly optimal (up to a log factor). Even for a simple stochastic gradient descent (SGD) algorithm, we need $\mO(\ve^{-1})$ sample to achieve $\ve$-optimal solution \citep{bottou2012stochastic}. The LQR problem is bilevel, with the critic part similar to SGD. Thus, the problem is more complicated than SGD and expects to require higher sample complexity. The convergence rate is also confirmed by the numerical examples below. 

\section{Proof sketch of the main theorem}\label{sec:proof}
In this section, we give a sketch of the proof of Theorem \ref{thm:main1} and postpone the details to the appendix. The lemmas used in the proof are stated in the later part of this section. 
\begin{proof}[Proof Sketch of Theorem \ref{thm:main1}]
First, we show in Lemma \ref{lem:critic_convex} that the critic loss is strongly convex.
Then, we show in Lemma \ref{lem:critic_gradient_accuracy} that we can obtain the sample of gradient with small bias: 
\begin{equation*}
    \left\| \EE \left[ \widehat{\nabla L_t}(\theta_t) - \nabla L_t(\theta_t) | \mG_t \right] \right\|_F \le \delta 
\end{equation*}
With these two lemmas, we show in Lemma \ref{lem:critic_improvement} that there is an improvement of critic error in each step:
\begin{multline}\label{eq:critic_imporvement_sketch}
     \quad \EE \left[ \| \theta_{t+1} - \theta_{K_{t+1}} \|^2_F | \mG_t \right] - \| \theta_t - \theta_{K_t} \|_F^2 \\
     \le - \frac43 \alpha_t \mu_{\sigma} \| \theta_t - \theta_{K_t} \|_F^2 + \frac14 \dfrac{\sigma_{min}(D_{\ep})}{c_3} \beta_t \ve + \bigl(\dfrac{3}{\alpha_t \mu_{\sigma}} + 2\bigr) \| \theta_{K_t} - \theta_{K_{t+1}} \|_F^2.
\end{multline}
Here, the term $\frac14 \frac{\sigma_{min}(D_{\ep})}{c_3} \beta_t \ve$ comes from the sample error in Lemma \ref{lem:critic_gradient_accuracy} and $(\frac{3}{\alpha_t \mu_{\sigma}} + 2) \| \theta_{K_t} - \theta_{K_{t+1}} \|_F^2$ is due to the actor update. Intuitively, we expect $\| \theta_{t+1} - \theta_{K_t} \|_F$ to be smaller than $\| \theta_t - \theta_{K_t} \|_F$, recall that $\| \theta_t - \theta_{K_t} \|_F$ measures the error of $\theta_t$ w.r.t. the current policy parameter $K_t$, while the last term in \eqref{eq:critic_imporvement_sketch} takes into account the update of $K_t$ to $K_{t+1}$ in the actor step. 
 
Furthermore, we establish the improvement of the actor in Lemma \ref{lem:actor_improvement}:
\begin{equation}\label{eq:actor_improvement_sktech}
\begin{aligned}
    & J(K_{t+1}) - J(K_t) \le -\beta_t \dfrac{\sigma_{min}(D_{\ep})}{c_3} (J(K_t) - J(K^*))\\
    & \quad - \beta_t  \left[\sigma_{min}(D_{\ep}) - \beta_t c_D (\|R\| + c_P \|B\|^2) \right] \|\widehat{G}_{K_t}\|_F^2 + \beta_t c_D \|G_{K_t} - \widehat{G}_{K_t}\|_F^2
\end{aligned}
\end{equation}
where the last term comes from the critic error.

To establish the convergence, we define a Lyapunov function
\begin{equation*}
    \mL_t = \mL(\theta_t, K_t) := \|\theta_t - \theta_{K_t}\|_F^2 + J(K_t) - J(K^*),
\end{equation*}
which is the sum of critic and actor errors. 
Direct computation shows that the last term in \eqref{eq:critic_imporvement_sketch} can be bounded by the second term in \eqref{eq:actor_improvement_sktech} and the last term in \eqref{eq:actor_improvement_sktech} can be bounded by $\frac14$ of the first term in \eqref{eq:critic_imporvement_sketch}. Therefore, combining \eqref{eq:critic_imporvement_sketch} and \eqref{eq:actor_improvement_sktech}, we obtain the decay estimate of the Lyapunov function 
\begin{equation}\label{eq:contraction_sketch}
\EE[\mL_{t+1} - \mL_t] \le - \EE \left[ \alpha_t \mu_{\sigma} \| \theta_t - \theta_{K_t} \|_F^2 + \beta_t \dfrac{\sigma_{min}(D_{\ep})}{c_3} (J(K_t) - J(K^*)) \right] + \frac14 \dfrac{\sigma_{min}(D_{\ep})}{c_3} \beta_t \ve.
\end{equation}
Notice that the last term (sample error) in \eqref{eq:contraction_sketch} can be bounded 
by the first term if $\EE [\| \theta_t - \theta_{K_t} \|_F^2] \ge \frac{\ve}{2}$ (according to the first inequality of \eqref{eq:stepsizes}) or by the second term if $\EE[J(K_t) - J(K^*)] \ge \frac{\ve}{2}$ and we will obtain a contraction rate for the Lyapunov function: 
$$\mL_{t+1} - \mL_t \le - \mO(\beta_t) \mL_t.$$
If both $\EE[\| \theta_t - \theta_{K_t} \|_F^2] < \frac{\ve}{2}$ and $\EE [J(K_t) - J(K^*)] < \frac{\ve}{2}$, then $\EE[\mL_t] < \ve$ and we can easily show that $\EE[\mL_{t+1}]$ is also less than $\ve$. This finishes the proof. 
\end{proof}

In summary, the key point of the proof is that we can bound the positive term in the critic improvement by the negative term in the actor improvement and vice versa. In this way, we obtain a contraction rate of the Lyapunov function.

Before we turn to the analysis of critic and actor parts, we state the following lemma which provides bounds for matrices $D_{K_t}$, $P_{K_t}$, and $\Sigma_{K_t}$. 
\begin{lemma}\label{lem:bound}
Under Assumption \ref{assump}, the matrix $D_{K_t}$, $P_{K_t}$ and $\Sigma_{K_t}$ satisfy
\begin{equation}\label{eq:bounds}
    \sigma_{min}(D_{\ep}) \le D_{K_t} \le c_D, ~~~
    P_{K_t} \le c_P, ~~ \text{and} ~~ 
    \Sigma_{K_t} \le  c_{\Sigma}
\end{equation}
where the three constants $c_D, c_P, c_{\Sigma} = \mO(1)$ only depend on $A$, $B$, $D_{\ep}$, $Q$, $R$, $\rho$, $\sigma$, and $c_A$.
Furthermore, the first inequality also holds with $D_{K_t}$ replaced by $D_{K^*}$.
\end{lemma}

\subsection{Analysis of the critic part}
In this subsection, we analyze the critic part of the algorithm. All the proofs are deferred to the appendix. 
Let us start with the following lemma, which gives the strong convexity property of the critic loss.
\begin{lemma}[Strong convexity of critic loss]\label{lem:critic_convex}
Suppose that $\rho(E) \le \rho < 1$, 
$L_K(\theta)$ is $\mu_{\sigma}$-strongly convex in $\theta$, where $\mu_{\sigma}>0$ only depends on $A$, $B$, $D_{\ep}$, $\rho$, $\sigma$, $c_K$, and $c_{\Sigma}$. Moreover, $\mu_{\sigma} = \mO(\sigma^4)$ when $\sigma$ is small. 
\end{lemma}
Actually, one technical reason of using a stochastic policy for exploration is to guarantee the strong convexity. 
The next lemma gives a quantitative description of the accuracy of critic gradient sampling proposed in \S\ref{sec:critic_alg}. 
\begin{lemma}[Gradient sample accuracy]\label{lem:critic_gradient_accuracy}
Under Assumption \ref{assump}, for any $\delta >0$ that is sufficiently small, let $\widehat{\nabla L_t}(\theta_t)$ be the sample of $\nabla L_t(\theta_t)$ with complexity $N, N_1 = \mO(1)$ and $N_0 = \mO(\log \frac{1}{\delta})$. Then, we have
\begin{equation}\label{eq:error_grad}
    \left\| \EE \left[ \widehat{\nabla L_t}(\theta_t) - \nabla L_t(\theta_t) ~\Big|~ \mG_t \right] \right\|_F \le \delta 
\end{equation}
and 
\begin{equation}\label{eq:critic_grad_bound}
    \EE \left[ \| \widehat{\nabla L_t}(\theta_t) \|^2_F ~\Big|~ \mG_t \right] \le c_L^2,
\end{equation}
where $c_L = \mO(1)$ is a positive constant that only depends on $A$, $B$, $D_{\ep}$, $Q$, $R$, $\sigma$, $c_K$, and $c_{\theta}$. 
\end{lemma}

\begin{rmk}
When we apply this lemma later, we will set $\delta^2 = \frac{1}{24} \frac{\sigma_{min}(D_{\ep})}{\kappa c_3} \mu_{\sigma} \ve$, and thus $\delta = \mO(\ve^{\frac12})$. By definition of the step sizes \eqref{eq:stepsize}, we have
\begin{equation}\label{eq:sample_error_bound}
    2 \alpha_t^2 \EE\left[\| \widehat{\nabla L_t}(\theta_t) \|_F^2 ~\Big|~ \mG_t\right] \le \frac{1}{8} \beta_t \dfrac{\sigma_{min}(D_{\ep})}{c_3} \ve.
\end{equation}
when \eqref{eq:critic_grad_bound} holds. This inequality \eqref{eq:sample_error_bound} will be used later and we can see that the step size has to be of order $\mO(\ve)$ to guarantee \eqref{eq:sample_error_bound}.
\end{rmk}
Next, we show a Lipschitz property for $\theta_K$ with respect to $K$.
\begin{lemma}\label{lem:thetaK}
For any two actor parameters $K$ and $K'$ such that $\|K\|, \|K'\| \le c_K$, $\|A-BK\|, \|A-BK'\| \le c_A$, and $\rho(A-BK), \rho(A-BK') \le \rho < 1$, we have
\begin{equation*}
    \| \theta_K - \theta_{K'} \|_F \le c_1 \|K-K'\|_F,
\end{equation*}
where the constant $c_1 = \mO(1)$ only depends on $A$, $B$, $R$, $\rho$, $c_A$, $c_K$, and $c_P$.
\end{lemma}

With the above lemmas, we can establish the improvement by the critic update.
\begin{lemma}\label{lem:critic_improvement}
Let the step size be defined as in \eqref{eq:stepsize} and Assumption \ref{assump} hold. For any $\ve > 0$ that is sufficiently small, assume that \eqref{eq:error_grad} and \eqref{eq:critic_grad_bound} hold with $\delta^2 = \frac{1}{24} \frac{\sigma_{min}(D_{\ep})}{\kappa c_3} \mu_{\sigma} \ve$ for all $t$, then we have 
\begin{multline}\label{eq:critic_imporvement0}
\EE \left[ \| \theta_{t+1} - \theta_{K_{t+1}} \|^2_F ~\big|~ \mG_t \right] - \| \theta_t - \theta_{K_t} \|_F^2 \\
     \le - \frac43 \alpha_t \mu_{\sigma} \| \theta_t - \theta_{K_t} \|_F^2 + \frac14 \dfrac{\sigma_{min}(D_{\ep})}{c_3} \beta_t \ve + \bigl(\dfrac{3}{\alpha_t \mu_{\sigma}} + 2\bigr) \| \theta_{K_t} - \theta_{K_{t+1}} \|_F^2.
\end{multline}
\end{lemma}
Recall that $K_{t+1}$ is $\mG_t$-measurable.

\subsection{Analysis of the actor part}
In this subsection, we give the convergence result for the actor part. All proofs are deferred to the appendix. 
The first lemma demonstrates that the cost functional is roughly quadratic in $G_K$. Inequality \eqref{eq:gradient_dominant} has also been established in earlier works \citep{fazel2018global, fu2020single}. 
\begin{lemma}\label{lem:gradient_dominant}
Let $K$ be an actor parameter such that $\rho(A-BK) < 1$, we have  
\begin{equation}\label{eq:gradient_dominant}
    c_2 \Tr(G_K G_K\tp) \le J(K) - J(K^*) \le c_3 \Tr(G_K G_K\tp),
\end{equation}
with positive constants $c_2 = \frac{\sigma_{min}(D_{\ep})}{\|R\| + c_P \|B\|^2}$ and $c_3 = \frac{\|D_{K^*}\|}{\sigma_{min}(R)}$. 
\end{lemma}

We recall that $\|\cdot\|$ denotes the operator norm of a matrix.
We also recall that $K^*$ is the optimal control parameter that is given by $K^* = (R + B\tp P^*B)^{-1} B\tp P^* A$ (see \eqref{eq:Riccati} for definition of $P^*$). Next lemma establishes the improvement of the actor update.

\begin{lemma}[Improvement in the actor update]\label{lem:actor_improvement}
Let the actor update be defined by \eqref{eq:actor_update} and Assumption \ref{assump} hold, then 
\begin{align*}
    & J(K_{t+1}) - J(K_t) \le -\beta_t \dfrac{\sigma_{min}(D_{\ep})}{c_3} (J(K_t) - J(K^*))\\
    & \quad - \beta_t \left[\sigma_{min}(D_{\ep}) - \beta_t c_D (\|R\| + c_P \|B\|^2) \right]  \|\widehat{G}_{K_t}\|_F^2 + \beta_t c_D \|G_{K_t} - \widehat{G}_{K_t}\|_F^2
\end{align*}
\end{lemma}
\begin{rmk}
This actor improvement lemma is a generalization of Lemma 15 in \cite{fazel2018global}. Their lemma shows an improvement of policy gradient with accurate critic, while our lemma shows that there are extra terms when we have stochastic estimate of the critic.
\end{rmk}

\section{Numerical Examples}\label{sec:numerical}
In this section, we present some numerical examples to validate our theoretical results. The code can be found at \cite{zhou2021actorcriticPDE-git}. 
We consider two examples: the first one has $d=2$ and $k=3$:
$$A = \begin{bmatrix} 0.5&0\\ 0&0.5 \end{bmatrix}, ~~ B = \begin{bmatrix} 0.2&0&0.1\\ 0&0.2&0.1 \end{bmatrix}, ~~ Q=\begin{bmatrix} 1&0\\ 0&0.8 \end{bmatrix}, ~~ R=\begin{bmatrix} 1&0&0\\ 0&1&0 \\ 0&0&0.5 \end{bmatrix}, ~~ D_{\xi} = \begin{bmatrix} 1&0\\ 0&1 \end{bmatrix},$$
and $\sigma=1$.
The other one has $d=4$ and $k=3$:
$$A = \begin{bmatrix} 0.5&0.1&0&0\\ 0.1&0.5&0.1&0\\ 0&0.1&0.5&0\\ 0&0&0&0.5 \end{bmatrix}, ~~ B = \begin{bmatrix} 0.3&0.1&0\\ 0.1&0.3&0.1\\ 0&0.1&0.3\\ 0.1&0.1&0.1 \end{bmatrix}, ~~ Q=\begin{bmatrix} 1&0&0&0\\ 0&1&0.1&0\\ 0&0.1&1&0.1\\ 0&0&0.1&1 \end{bmatrix}, $$
$$R=\begin{bmatrix} 1&0.1&0\\ 0.1&1&0.1 \\ 0&0.1&1 \end{bmatrix}, ~~ D_{\xi} = \begin{bmatrix} 1&0&0.1&0\\ 0&1&0&0\\ 0.1&0&1&0.1\\ 0&0&0.1&1 \end{bmatrix},$$
and $\sigma=1$. In all the tests, we set $N = N_0 = N_1 = 100$ for simplicity. We test for $T=125, 250, 500, 1000, 2000, 4000$. In each example, we set the step sizes to be $\alpha_t = \beta_t = \frac{4}{T}$. In order to save time, we multiply the step sizes by 3 for the first $T/2$ steps.

Figure \ref{fig:learning_curve} shows the learning curves for the two example with step size $\alpha_t=\beta_t=0.001$. The error is the average of $10$ independent runs, and we also show the standard deviations. In the beginning, the error curves are nearly straight lines, which coincide with our one-step improvement analysis in the previous section. Then the errors become static because the algorithm has reached its capacity.

\begin{figure*}[t!]
    \centering
    \includegraphics[width=0.49\textwidth]{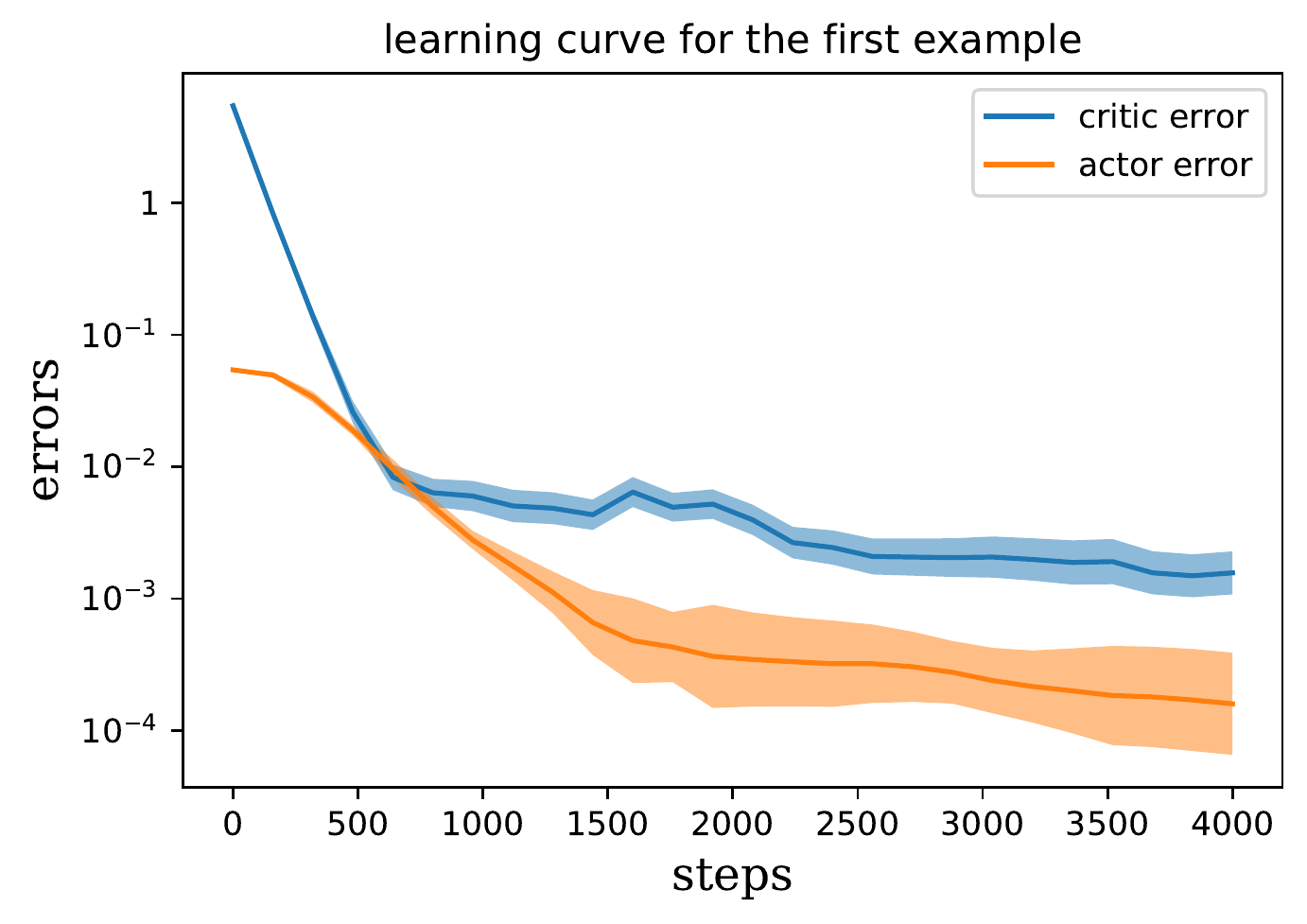}
    \includegraphics[width=0.49\textwidth]{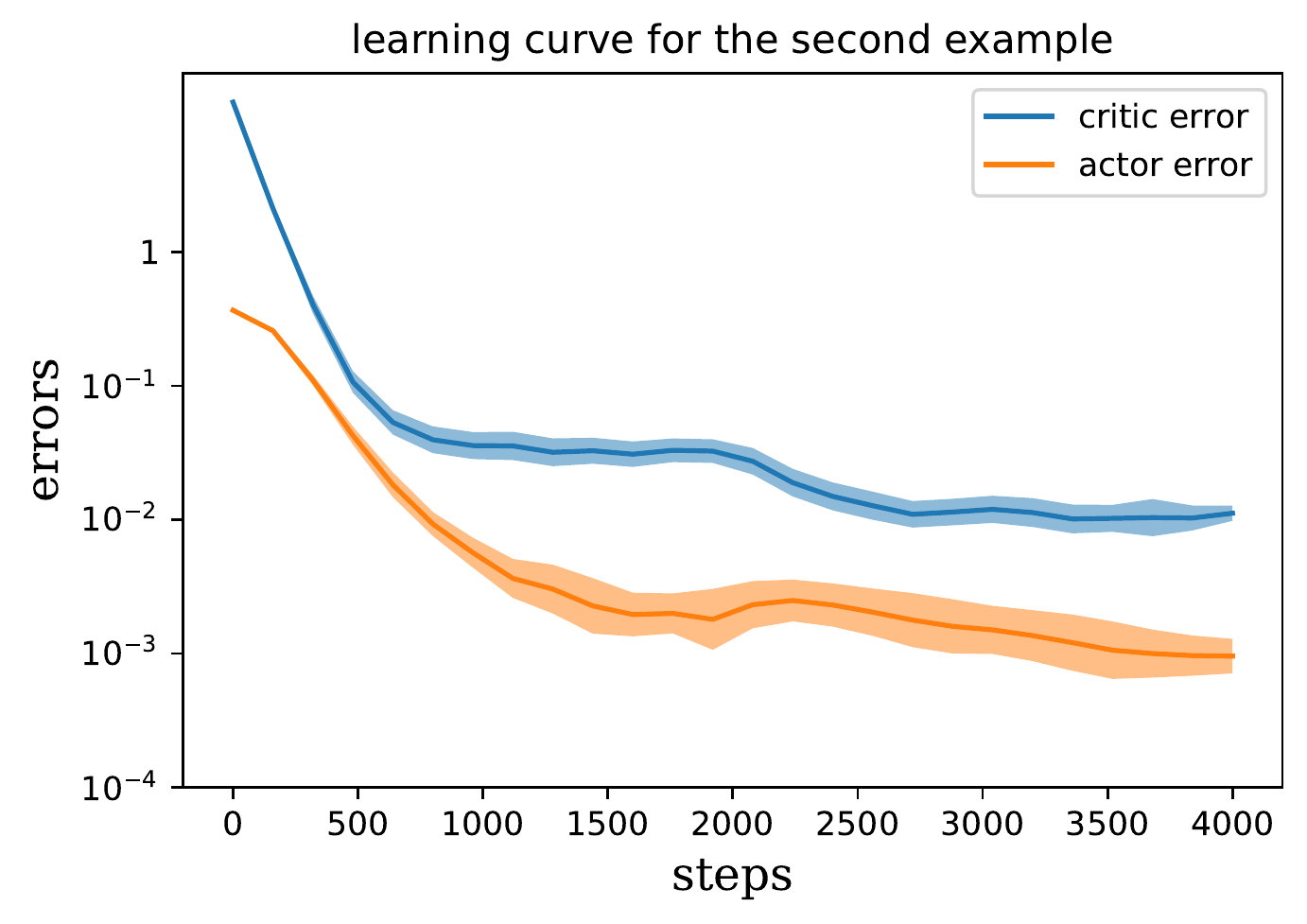}
    \caption{The error curves for the two examples with step size $\alpha_t=\beta_t=0.001$. The errors are the average of $10$ independent runs, with standard deviation plotted.}
    \label{fig:learning_curve}
\end{figure*}

In order to obtain a convergence rate, we also test different step sizes, which is shown in Figure \ref{fig:convergence_rate}. In the tests, we keep $T \alpha_t = T \beta_t$ as a constant. The horizontal axis marks the number of steps $T$, ranging from 125 to 4000. We take a $log_2$ transform of $T$. The vertical axis is the final critic and actor errors (after a $log_2$ transform). A linear regression indicates that the slopes of the four error curves are all $-1.0$, which confirms our theoretical results in the previous section.

\begin{figure*}[t!]
    \centering
    \includegraphics[width=0.49\textwidth]{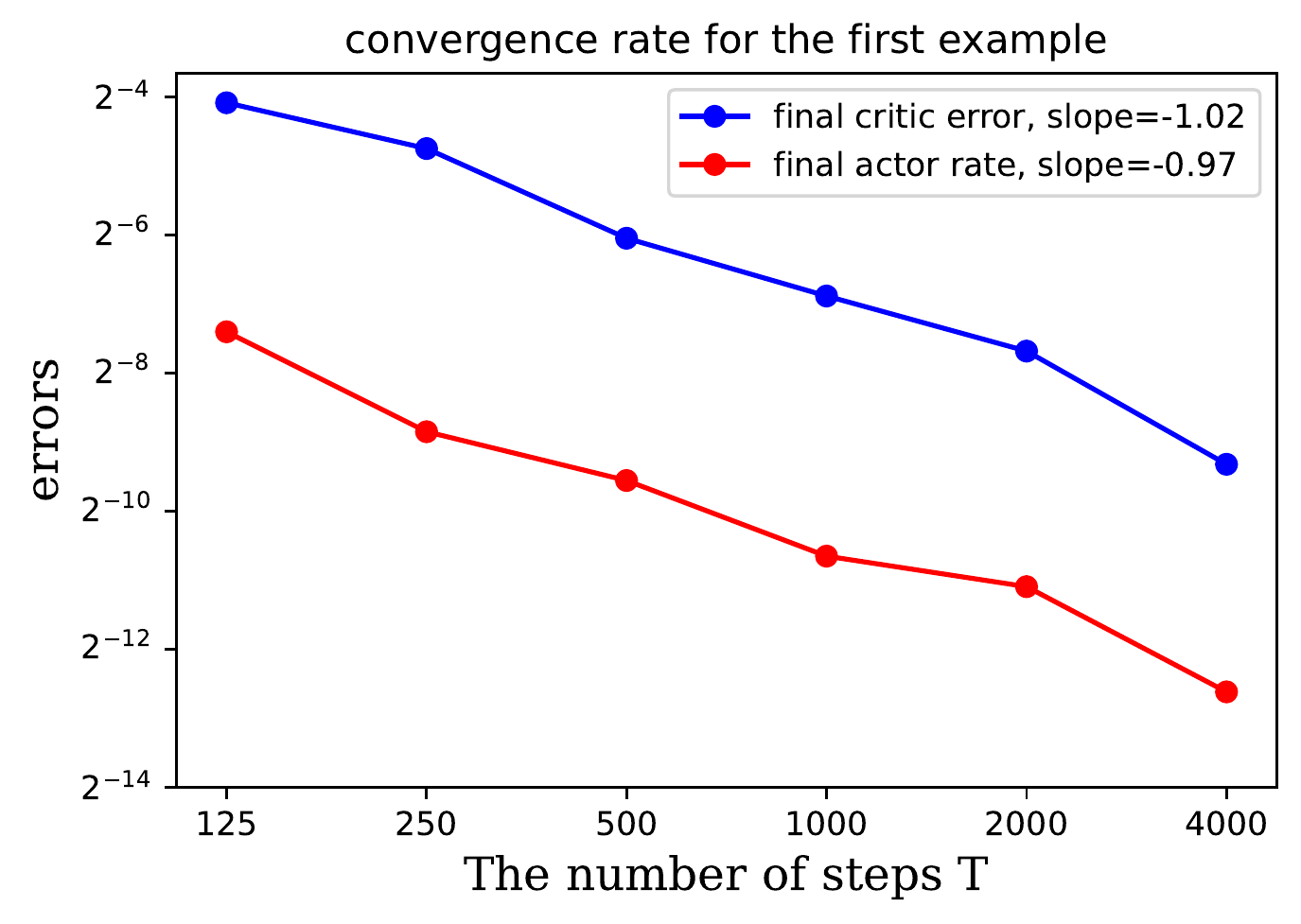}
    \includegraphics[width=0.49\textwidth]{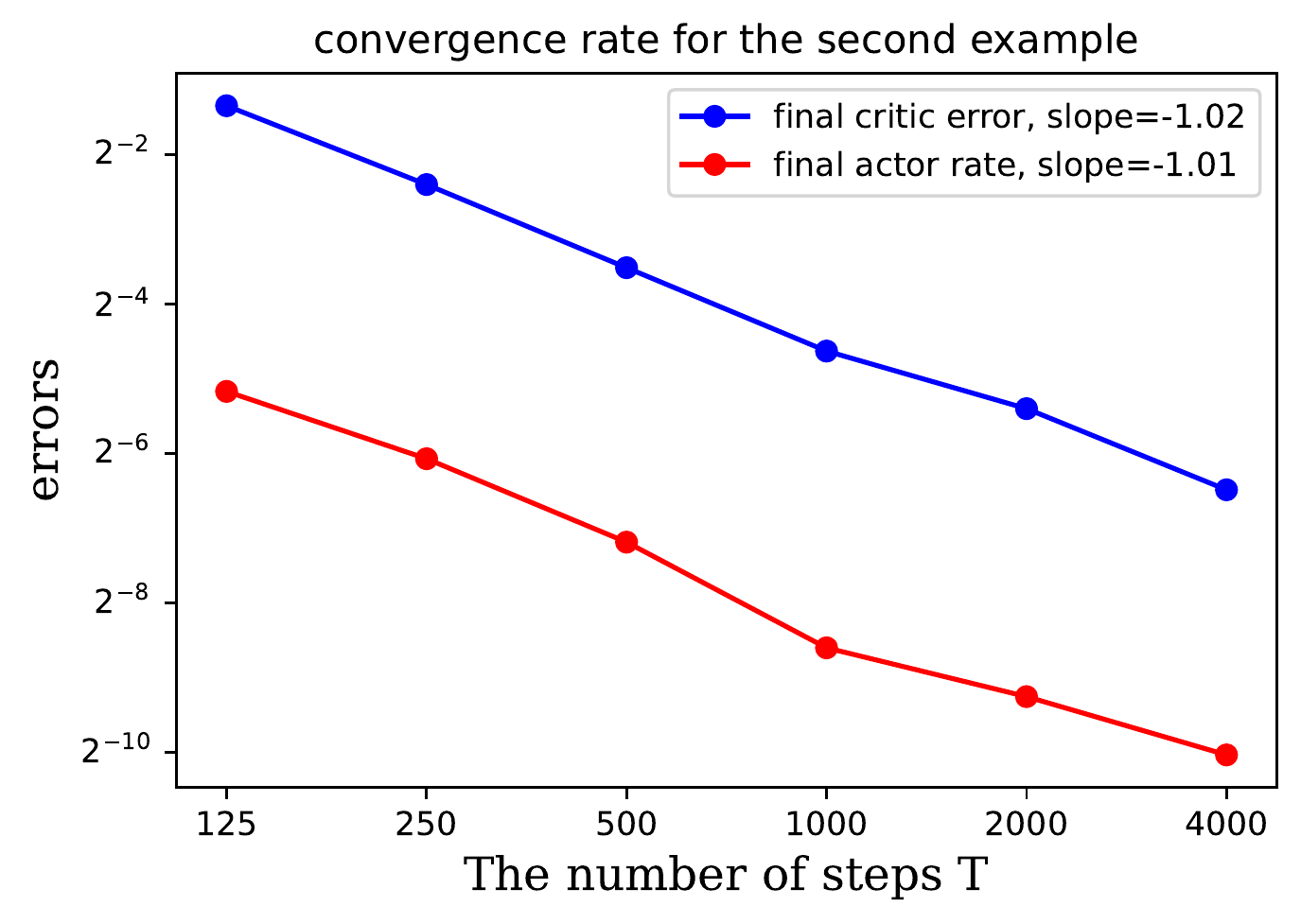}
    \caption{The convergence rate for the two examples with the numbers of steps ranging from $T=125$ to $T=4000$ and step size $\frac{4}{T}$. Each error is the average of $10$ independent runs. The slope for the four error curves are all $-1.0$.}
    \label{fig:convergence_rate}
\end{figure*}

We also track the norm in Assumption \ref{assump}. In the numerical tests, the maximum of $\rho(A-BK_t)$, $\|A-BK_t\|$, $\|E_t\|$, $\|K_t\|$, and $\|\theta_t\|_F$ for the first and second examples are $0.524$, $0.529$, $0.586$, $0.329$, $2.641$ and $0.662$, $0.662$, $0.867$, $0.498$, $4.254$ respectively. This further confirms that Assumption \ref{assump} is reasonable.


\acks{This work is supported in part by the National Science Foundation via grants DMS-2012286 and CCF-1934964 (Duke TRIPODS).}


\newpage

\appendix
\section{Proofs}
Throughout the proof, we will frequently use two basic properties in linear algebra. So we state them here. The first one is that if $X$ is a (symmetric and) positive semi-definite matrix and $Y$ is of the same shape, then $\Tr(XY) \le \Tr(X) \|Y\|$, where we recall that $\|\cdot\|$ is the operator norm of a matrix. The second property is a direct corollary of the first one: for any matrices $X$ and $Y$ of proper shapes, we have $\|XY\|_F \le \|X\| \, \|Y\|_F$
\subsection{Proofs for results in Section \ref{sec:background} and Section \ref{sec:algorithm}}
\begin{proof}[Proof of Proposition \ref{prop:cost_expression}]
Since $\rho(A-BK) < 1$, we know from definition \eqref{eq:PK} that the expression for $P_K$ in series is
\begin{equation}\label{eq:PK_series}
    P_K = \sum_{s=0}^{\infty} ((A-BK)\tp)^s (Q + K\tp R K) (A-BK)^s.
\end{equation}
Give the state $x_s$, the conditional expectation of one-step cost is
\begin{equation}\label{eq:one_step_cost}
\begin{aligned}
    \EE[c(x_s, u_s) | x_s] &= x_s\tp Q x_s + \EE_{\omega_s \sim N(0, Id)} [(-K x_s + \sigma \omega_s)\tp R (-K x_s + \sigma \omega_s)] \\
    &= x_s\tp (Q + K\tp R K) x_s + \sigma^2 \Tr(R).
\end{aligned}
\end{equation}
So the total cost is
\begin{equation*}
\begin{aligned}
     J(K) & = \lim_{S \to \infty} \EE_K \left[ \dfrac{1}{S} \sum_{s=0}^{S-1} c(x_s, u_s) \right] = \lim_{S \to \infty} \EE_K \left[ \dfrac{1}{S} \sum_{s=0}^{S-1} \EE[c(x_s, u_s) | x_s] \right] \\
     & = \lim_{S \to \infty} \EE_K \left[ \dfrac{1}{S} \sum_{s=0}^{S-1} x_s\tp (Q + K\tp R K) x_s\right] + \sigma^2 \Tr(R) \\
     & = \EE_K[x\tp (Q + K\tp R K) x] + \sigma^2 \Tr(R) \\
     & = \Tr\left[ \EE_K[xx\tp] (Q + K\tp R K) \right] + \sigma^2 \Tr(R) = \Tr\left[ D_K (Q + K\tp R K) \right] + \sigma^2 \Tr(R) \\
     & = \Tr\left[ D_K (P_K - (A-BK)\tp P_K (A-BK)) \right] + \sigma^2 \Tr(R) \\
     & = \Tr\left[ (D_K - (A-BK) D_K (A-BK)\tp) P_K \right] + \sigma^2 \Tr(R) = \Tr[D_{\ep} P_K] + \sigma^2 \Tr(R).
\end{aligned}
\end{equation*}
So \eqref{eq:JK} holds. Next, we derive the expression for $\nabla_K J(K)$. We need a simple formula: if the shape of $M$ is the same as the shape of $K$, then $\nabla_K \Tr(M\tp K) = \nabla_K \Tr(M K\tp) = M$. Since $J(K) = \Tr\left[ D_K (Q + K\tp R K) \right] + \sigma^2 \Tr(R)$, we have
\begin{equation}\label{eq:nabla_J}
    \nabla_K J(K) = 2RK D_K + \nabla_K \Tr[D_K Q_0] |_{Q_0 = Q + K\tp R K}.
\end{equation}
We recall that 
$$D_K = D_{\ep} + (A-BK) D_K (A-BK)\tp.$$
Therefore, 
\begin{equation}\label{eq:recursion}
\begin{aligned}
     & \quad \nabla_K \Tr[D_K Q_0] = \nabla_K \Tr[(D_{\ep} + (A-BK) D_K (A-BK)\tp) Q_0] \\
     & = -B\tp (Q_0 + Q_0\tp) (A-BK) D_K + \nabla_K \Tr[D_K Q_1] |_{Q_1 = (A-BK)\tp Q_0 (A-BK)} \\
     & = -2B\tp Q_0 (A-BK) D_K + \nabla_K \Tr[D_K Q_1] |_{Q_1 = (A-BK)\tp Q_0 (A-BK)}
\end{aligned}
\end{equation}
where we used $Q_0 = Q_0\tp$ in the last equality. Therefore, we can apply \eqref{eq:recursion} recursively and obtain
\begin{equation}\label{eq:recursion2}
\begin{aligned}
     & \quad \nabla_K \Tr[D_K Q_0] |_{Q_0 = Q + K\tp R K} \\
     & = -2B\tp (Q + K\tp R K) (A-BK) D_K + \nabla_K \Tr[D_K Q_1] |_{Q_1 = (A-BK)\tp (Q + K\tp R K) (A-BK)} \\
     & = -2B\tp (Q + K\tp R K) (A-BK) D_K -2B\tp (A-BK)\tp (Q + K\tp R K) (A-BK)^2 D_K \\
     & \quad + \nabla_K \Tr[D_K Q_2] |_{Q_2 = ((A-BK)\tp)^2 (Q + K\tp R K) (A-BK)^2} \\
     & = \cdots \\
     & = -\sum_{s=0}^{\infty} 2B\tp ((A-BK)\tp)^s (Q + K\tp R K) (A-BK)^{s+1} D_K \\
     & = -2 B\tp P_K (A-BK) D_K
\end{aligned}
\end{equation}
where the assumption $\rho(A-BK) < 1$ guarantees that the series converges and the remaining term vanishes.
Substituting \eqref{eq:recursion2} into \eqref{eq:nabla_J}, we obtain
$$\nabla_K J(K) = 2RK D_K - 2 B\tp P_K (A-BK) D_K = 2 \left[ (R + B\tp P_K B)K - B\tp P_K A \right] D_K.$$
\end{proof}

\begin{proof}[Proof of Proposition \ref{prop:value_expression}]
If we start with $x_0 = x$, since the state dynamic is
$$x_{s+1} = (A -BK) x_s + \ep_s$$
with $\ep_s \sim N(0, D_{\ep})$, the state distribution is
$$x_s \sim N\left( (A-BK)^s x,~ \sum_{i=0}^{s-1} (A-BK)^i D_{\ep} ((A-BK)\tp)^i \right) =: N\left((A-BK)^s x,~ D^{(s)}_K\right).$$
Therefore, by definition, the value function is
\begin{equation*}
\begin{aligned}
     V_K(x) & = \sum_{s=0}^{\infty} \left\{ \EE_K \left[ c(x_s, u_s) ~|~ x_0=x \right] - J(K) \right\} \\
     & = \sum_{s=0}^{\infty} \left\{ \EE_K \left[ x_s\tp (Q + K\tp R K) x_s ~|~ x_0=x \right] + \sigma^2 \Tr(R) - J(K) \right\} \\
     & = \sum_{s=0}^{\infty} \left\{ \Tr\left( \EE_K \left[x_s x_s\tp  ~|~ x_0=x \right] (Q + K\tp R K) \right) - \Tr[D_{\ep} P_K] \right\} \\
     & = \sum_{s=0}^{\infty} \left\{ \Tr\left[ \left((A-BK)^s xx\tp ((A-BK)\tp)^s + D^{(s)}_K \right) (Q + K\tp R K) \right]  - \Tr[D_{\ep} P_K] \right\},
\end{aligned}
\end{equation*}
where the second equality is by \eqref{eq:one_step_cost}, the third equality is by \eqref{eq:JK}.
Therefore,
\begin{equation*}
\begin{aligned}
     & \quad V_K(x) \\
     & = x\tp P_K x + \sum_{s=0}^{\infty} \left\{  \Tr\left[ \left(\sum_{i=0}^{s-1} (A-BK)^i D_{\ep} ((A-BK)\tp)^i \right) (Q + K\tp R K) \right] \right.\\
     & \hspace{1.1in} \left.- \Tr\left[D_{\ep} \left( \sum_{i=0}^{\infty} ((A-BK)\tp)^i (Q + K\tp R K) (A-BK)^i \right)\right] \right\} \\
     & = x\tp P_K x - \sum_{s=0}^{\infty} \Tr\left[ \left(\sum_{i=s}^{\infty} (A-BK)^i D_{\ep} ((A-BK)\tp)^i \right) (Q + K\tp R K) \right]\\
     & = x\tp P_K x - \sum_{s=0}^{\infty} \sum_{j=0}^{\infty} \Tr\left[ \left( (A-BK)^s D_{\ep} ((A-BK)\tp)^s \right) \right. \\
     & \hspace{1.5in} \left. \left( ((A-BK)\tp)^j (Q + K\tp R K) (A-BK)^j \right)  \right]\\
     & = x\tp P_K x - \sum_{s=0}^{\infty} \left\{  \Tr\left[ \left( (A-BK)^s D_{\ep} ((A-BK)\tp)^s \right) P_K \right] \right\} =  x\tp P_K x - \Tr[D_K P_K],
\end{aligned}
\end{equation*}
where we have used the series expressions for $P_K$ \eqref{eq:PK_series} and $D_K$ \eqref{eq:DK_series}. The assumption $\rho(A-BK)<1$ guarantees that all the series above converge. Next, we compute the state-value function $Q_K(x, u)$. Recall that $Q_K(x,u)$ is the expected extra cost if we start at $x_0=x$, take a first action $u_0=u$ and then follow the policy $\pi_K$. Therefore,
\begin{equation*}
\begin{aligned}
     & \quad Q_K(x, u) = c(x,u) - J(K) + \EE[V_K(x') ~|~ x,u] \\
     & = x\tp Q x + u\tp R u - \Tr[D_{\ep} P_K] - \sigma^2 \Tr(R) + \EE_{x' \sim N(Ax+Bu, D_{\xi})}[x\ptp P_K x' - \Tr[D_K P_K]] \\
     & = x\tp Q x + u\tp R u - \Tr[D_{\ep} P_K] - \sigma^2 \Tr(R) + \Tr\left[\EE_{x' \sim N(Ax+Bu, D_{\xi})}[x' x\ptp ] P_K \right] - \Tr[D_K P_K] \\
     & = x\tp Q x + u\tp R u - \Tr[D_{\ep} P_K + \sigma^2 R + D_K P_K] + \Tr\left[ \left((Ax+Bu) (Ax+Bu)\tp + D_{\xi}\right) P_K \right] \\
     & = x\tp Q x + u\tp R u - \Tr[(D_{\ep} - D_{\xi}) P_K + \sigma^2 R + D_K P_K] + (Ax+Bu)\tp P_K (Ax+Bu)\\
     & = \begin{bmatrix}  x\tp & u\tp \end{bmatrix} \begin{bmatrix} Q + A\tp P_K A & A\tp P_K B \\ B\tp P_K A & R + B\tp P_K B \end{bmatrix} \begin{bmatrix}  x \\ u \end{bmatrix} - \sigma^2 \Tr(R + P_K BB\tp) - \Tr(D_K P_K).
\end{aligned}
\end{equation*}
\end{proof}

\begin{proof}[Proof of Proposition \ref{prop:rescaled_gradient}]
The distribution of policy is $\pi_K(u | x) \sim N(-Kx, \sigma^2 I_k)$, with probability density 
$$\pi_K(u | x) = (2\pi \sigma^2)^{-k/2} \exp\left( - \frac{1}{2\sigma^2} |u + Kx|^2 \right).$$
Therefore,
$$\log \pi_K(u | x) = -\frac{k}{2} \log(2\pi \sigma^2) - \frac{1}{2\sigma^2} |u + Kx|^2$$
and
$$\nabla_K \log \pi_K(u | x) =  - \frac{1}{\sigma^2} (u + Kx) x\tp.$$
Therefore, by the definition in \eqref{eq:Fisher_state}, the Fisher information matrix at state $x$ is
\begin{equation*}
\begin{aligned}
     F_x(K) &= \int_{\RR^k} (2\pi \sigma^2)^{-k/2} \exp\left( - \frac{1}{2\sigma^2} |u + Kx|^2 \right) \frac{1}{\sigma^4} [(u + Kx) x\tp] \otimes [(u + Kx) x\tp] \rd u\\
     & = \int_{\RR^k} (2\pi \sigma^2)^{-k/2} \exp\left( - \frac{1}{2\sigma^2} |u|^2 \right) \frac{1}{\sigma^4} [u x\tp] \otimes [u x\tp] \rd u.
\end{aligned}
\end{equation*}
Recall that the stationary state distribution is $N(0, D_K)$. Hence, the Fisher information matrix is
\begin{equation*}
\begin{aligned}
     & \quad F(K) = \int_{\RR^d} (2\pi)^{-d/2} (\det(D_K))^{-1/2} \exp\left( -\frac12 x\tp D_K^{-1} x \right) F_x(K) \rd x \\
     & = \int_{\RR^d} (2\pi)^{-d/2} (\det(D_K))^{-1/2} \exp\left( -\frac12 x\tp D_K^{-1} x \right) \\ 
     & \quad \int_{\RR^k} (2\pi \sigma^2)^{-k/2} \exp\left( - \frac{1}{2\sigma^2} |u|^2 \right) \frac{1}{\sigma^4} [u x\tp] \otimes [u x\tp] \rd u \rd x
\end{aligned}
\end{equation*}
Note that we can compute the integration w.r.t. $x$ and $u$ separately with
$$\int_{\RR^d} (2\pi)^{-d/2} (\det(D_K))^{-1/2} \exp\left( -\frac12 x\tp D_K^{-1} x \right) xx\tp \rd x = D_K$$
and
$$\int_{\RR^k} (2\pi \sigma^2)^{-k/2} \exp\left( - \frac{1}{2\sigma^2} |u|^2 \right) uu\tp \rd u = \sigma^2 I_k.$$
Therefore, by an elementwise analysis, we obtain
\begin{equation*}
    \sigma^2 F(K) \cdot G_K = G_K D_K.
\end{equation*}
Therefore, \eqref{eq:rescaled_gradient} holds.
\end{proof}

\subsection{Proofs for results in section 4}
We first prove the lemmas and then the main theorem \ref{thm:main1}.

\begin{proof}[Proof of Lemma \ref{lem:bound}]
Firstly
$$D_{K_t} = D_{\ep} + (A-BK_t) D_{K_t} (A-BK_t)\tp \ge D_{\ep} \ge \sigma_{min}(D_{\ep}).$$
$D_{K_t}$ also has an expression in series:
$$D_{K_t} = \sum_{s=0}^{\infty} (A-BK_t)^s D_{\ep} ((A-BK_t)\tp)^s.$$
Since $\lim_{k \to \infty} \| (A- BK_t)^k \|^{\frac1k} = \rho(A-BK_t) \le \rho < 1$ and $\|A-BK_t\| \le c_A$, (with an argument similar to the proof in Lemma \ref{lem:critic_convex} below,) we have
$$D_{K_t} = \sum_{s=0}^{\infty} (A-BK_t)^s D_{\ep} ((A-BK_t)\tp)^s \lesssim \dfrac{1}{1- \rho^2} \|D_{\ep}\|$$
with the constant depending on $c_A$ and $d$. Therefore, the first inequality in \eqref{eq:bounds} holds. The constant $c_D$ is proportional to $\dfrac{1}{1- \rho^2} \|D_{\ep}\|$ and also depends on $c_A$ and $d$.
The argument above also holds for $K^*$, so the inequality also holds with $K_t$ replaced by $K^*$. For $P_{K_t}$, we also have an expression in series:
\begin{equation*}
    P_{K_t} = \sum_{s=0}^{\infty} ((A-BK_t)\tp)^s (Q + B\tp R B) (A-BK_t)^s.
\end{equation*}
So the argument to prove the second inequality of \eqref{eq:bounds} is the same. Finally, since $\Sigma_{K_t}$ has expression \eqref{eq:stationary_z} with $\|D_{K_t}\| \le c_D$ and $\|K_t\| \le c_K$, $\|\Sigma_{K_t}\|$ has a bound $c_{\Sigma} = (1+c_K)^2 c_D + \sigma^2$ automatically.
\end{proof}

\subsubsection{Proofs for critic}
Here we prove the results for the critic.

\begin{proof}[Proof of Lemma \ref{lem:critic_convex}]
In order to show $\nabla^2 L_K(\theta) =  \EE_K \left[ \psi(x, u) \otimes \psi(x, u) \right] \ge \mu_{\sigma}$, we only need to show that for any $M \in \RR^{(d+k) \times (d+k)}$, we have
$$\EE_K \left[ (\Tr[M \psi(x, u)])^2 \right] \ge \mu_{\sigma} \|M\|_F^2.$$
Since $\psi(x,u)$ is symmetric, we have $\Tr[M \psi(x, u)] = \Tr[M\tp \psi(x, u)] = \Tr[\frac12(M + M\tp) \psi(x, u)]$. We also have $2 \| \frac12(M + M\tp) \|_F^2 \ge \|M\|_F^2$. Therefore, we only need to show 
\begin{equation}\label{eq:goal_convex}
    \EE_K \left[ (\Tr[M \psi(x, u)])^2 \right] \ge 2\mu_{\sigma} \|M\|_F^2
\end{equation}
for all symmetric matrix $M$.
Recall that 
\begin{equation*}
    z_{s+1} = E z_s + \widetilde{\ep}_s.
\end{equation*}
Since
$$\psi(z) = \EE_K[(Ez + \widetilde{\ep})(Ez + \widetilde{\ep})\tp] - zz\tp = E zz\tp E\tp + \Sigma_{\ep} - zz\tp,$$
we have
$$\Tr[M \psi(x, u)] = \Tr[ME zz\tp E\tp + M\Sigma_{\ep} - Mzz\tp] = z\tp (E\tp M E - M) z + \Tr[M\Sigma_{\ep}].$$
Recall that $z \sim N(0 ,\Sigma_K)$ under the stationary distribution where $\Sigma_K$ is defined in \eqref{eq:stationary_z}. By definition, for any $x \in \RR^d$, $u \in \RR^k$, and $\gamma \neq 0$, we have
\begin{multline}
    \begin{bmatrix}  x\tp & u\tp \end{bmatrix} \Sigma_K \begin{bmatrix}  x \\ u \end{bmatrix} = (\gamma x - \frac{1}{\gamma} K\tp u)\tp D_K (\gamma x - \frac{1}{\gamma} K\tp u)\\
    + (1 - \gamma^2) x\tp D_K x + u\tp [\sigma^2 I_k - (\frac{1}{\gamma^2} - 1) K D_K K\tp] u.
\end{multline}
Therefore, we can smartly choose a $\gamma \in (0,1)$ s.t. $(1 - \gamma^2)D_K \ge \mu_{\Sigma}$ and $\sigma^2 I_k - (\frac{1}{\gamma^2} - 1) K D_K K\tp \ge \mu_{\Sigma}$ for some positive constant $\mu_{\Sigma} \in \RR$. Therefore, $\Sigma_K \ge \mu_{\Sigma}$.
Using the same method, we can also show that $\Sigma_{\ep} \ge \mu_{\Sigma}$.
This $\mu_{\Sigma}$ depends on $\sigma$, $\sigma_{min}(D_K)$ ($\sigma_{min}(D_{\ep})$ for $\Sigma_{\ep}$) and $\|K\|$. Since $\sigma_{min}(D_K) \ge \sigma_{min}(D_{\ep}) = \mO(1)$, $\mu_{\Sigma}$ is of order $\mO(1)$ as long as we have an upper bound for $\|K\|$. We can also find that $\mu_{\Sigma} = \mO(\sigma^2)$ when $\sigma$ is small.
Next, we start to compute \eqref{eq:goal_convex}.
\begin{equation}\label{eq:trace_square}
\begin{aligned}
     & \quad \EE_K \left[ (\Tr[M \psi(x, u)])^2 \right]\\
     & = \EE_K \left[ \left( z\tp (E\tp M E - M) z + \Tr[M\Sigma_{\ep}] \right) \left( z\tp (E\tp M E - M) z + \Tr[M\Sigma_{\ep}] \right) \right]\\
     & = \EE_K \left[ z\tp (E\tp M E - M) z z\tp (E\tp M E - M) z + 2 z\tp (E\tp M E - M) z \Tr[M\Sigma_{\ep}] + \Tr[M\Sigma_{\ep}]^2\right].
\end{aligned}
\end{equation}
We will compute each term respectively. We recall the stationary distribution is $z \sim N(0, \Sigma_K)$. If we define $w = \Sigma_K^{-\frac12} z$, then $w \sim N(0, I_{d+k})$. Denote $(m_{ij}) = \widetilde{M} = \Sigma_K^{\frac12} (E\tp M E - M) \Sigma_K^{\frac12}$, then $\widetilde{M}$ is symmetric and
\begin{equation}\label{eq:trace_term1}
\begin{aligned}
     & \quad \EE_K \left[ z\tp (E\tp M E - M) z z\tp (E\tp M E - M) z\right]\\
     & = \EE_{w \sim N(0, I_{d+k})} \left[ w\tp \Sigma_K^{\frac12} (E\tp M E - M) \Sigma_K^{\frac12} ww\tp \Sigma_K^{\frac12} (E\tp M E - M) \Sigma_K^{\frac12} w \right] \\
     & = \EE_{w \sim N(0, I_{d+k})} \left[ w\tp \widetilde{M} ww\tp \widetilde{M} w \right] \\
     & = \int_{\RR^{d+k}} (2\pi)^{-\frac{d+k}{2}} w\tp \widetilde{M} ww\tp \widetilde{M} w \exp\left( -\dfrac{|w|^2}{2} \right) \rd w\\
     & = 3 \sum_{i=1}^{d+k} m_{ii}^2 + \sum_{i \neq j} m_{ii} m_{jj} + 2\sum_{i \neq j} m_{ij}^2 = 2\Tr[\widetilde{M}^2] + \Tr[\widetilde{M}]^2 \\
     & = 2\Tr\left[ \Sigma_K (E\tp M E - M) \Sigma_K (E\tp M E - M) \right] + \Tr\left[ \Sigma_K (E\tp M E - M) \right]^2.
\end{aligned}
\end{equation}
Also,
\begin{equation}\label{eq:trace_term2}
    \EE_K \left[z\tp (E\tp M E - M) z \right] = \EE_K \left[ \Tr(z z\tp (E\tp M E - M)) \right] = \Tr[\Sigma_K (E\tp M E - M))].
\end{equation}
Recall that $\Sigma_K = \Sigma_{\ep} + E \Sigma_K E\tp$, so
\begin{equation}\label{eq:trace_term}
    \Tr[\Sigma_K (E\tp M E - M))] = -\Tr[M (\Sigma_K - E \Sigma_K E\tp)] = -\Tr[M \Sigma_{\ep}]
\end{equation}
Therefore, substituting \eqref{eq:trace_term1}, \eqref{eq:trace_term2} and \eqref{eq:trace_term} into \eqref{eq:trace_square}, we obtain
\begin{equation}\label{eq:trace_hessian}
\begin{aligned}
     & \quad \EE_K \left[ (\Tr[M \psi(x, u)])^2 \right] \\
     & = 2\Tr\left[ \Sigma_K (E\tp M E - M) \Sigma_K (E\tp M E - M) \right] + \Tr\left[ M \Sigma_{\ep}\right]^2 - 2\Tr\left[ M \Sigma_{\ep}\right]^2 + \Tr\left[ M \Sigma_{\ep}\right]^2 \\
     & = 2\Tr\left[ \Sigma_K (E\tp M E - M) \Sigma_K (E\tp M E - M) \right] \\
     & \ge 2\mu_{\Sigma} \Tr\left[ (E\tp M E - M) \Sigma_K (E\tp M E - M) \right] \\
     & \ge 2\mu_{\Sigma}^2 \, \| E\tp M E - M \|_F^2
\end{aligned}
\end{equation}
for all symmetric matrix $M$.
Next, we want to show $\|M\|_F \lesssim \|E\tp M E - M\|_F$. Since the Frobenius norm is equivalent to the operator norm (with the constant depending on the dimension), we only need to show $\|M\| \lesssim \|E\tp M E - M\|$. Note that this step makes $\mu_{\sigma}$ depend polynomially on $d+k$. We define an operator $\mT_E: \RR^{(d+k) \times (d+k)} \to \RR^{(d+k) \times (d+k)}$ such that
$$\mT_E(X) = \sum_{s=0}^{\infty} (E\tp)^s X E^s.$$
Since $1 > \rho \ge \rho(E) = \lim_{s \to \infty} \|E^s\|^{\frac1s}$, the norm of the operator should satisfy
$$\|\mT_E\| = sup_{X \neq 0} \dfrac{\|\mT_E(X)\|}{\|X\|} \le \dfrac{c}{1-\rho^2}$$
where $c$ depends on $\|E\|$ and $d+k$.
Notice that
$$\mT_E(M - E\tp M E) = \sum_{s=0}^{\infty} (E\tp)^s (M - E\tp M E) E^s = M,$$
we conclude that
$$\|M\| = \|\mT_E(M - E\tp M E)\| \le \|\mT_E\| \|M - E\tp M E\| \le \frac{c}{1-\rho^2} \|M - E\tp M E\|.$$
So, $\|M\|_F \lesssim \|E\tp M E - M\|_F$.
Therefore, by \eqref{eq:trace_hessian}, $\nabla^2 L_K(\theta) =  \EE_K \left[ \psi(x, u) \otimes \psi(x, u) \right] \ge \mu_{\sigma}$ holds with $\mu_{\sigma}$ proportional to $\sigma^4 / (1 - \rho^2)$ and depending on $\|E\|$ and $d+k$. Moreover, $\mu_{\sigma}$ grows polynomially as $d+k$ becomes large.
\end{proof}

\begin{proof}[Proof of Lemma \ref{lem:critic_gradient_accuracy}]
Similar to \eqref{eq:f}, we define
$$\nabla L_t(\theta_t) = \EE_{K_t} [f(x,u)],$$
where $f$ depends on both $\theta_t$ and $K_t$. We denote $\EE_{N_0} [f(x,u)]$ the expectation of the same function under the distribution of $(x_{N_0}, u_{N_0})$, which starts at $x_0 = 0$ and follows the policy $\pi_{K_t}$.
We prove \eqref{eq:critic_grad_bound} first. We recall that the feature matrix $\phi(x,u)$ defined in \eqref{eq:feature} is quadratic in $(x,u)$. So, $\psi(x, u) = \EE\left[ \phi(x', u') | x, u \right] - \phi(x,u)$ also grows at most quadratically in $(x,u)$ since $(x', u')$ are normally distributed. Therefore, $f(x,u)$, defined in \eqref{eq:f} grows at most quartically in $(x,u)$. By assumption \ref{assump}, $\|\theta_t\|_F \le c_{\theta} = \mO(1)$ and $\|K_t\| \le c_K = \mO(1)$, so the coefficients for this quadratic growth are of order $\mO(1)$.
A similar argument tells us that $\widehat{f}(x^{(i)}_{N_0}, u^{(i)}_{N_0})$ defined in \eqref{eq:f_hat} grows at most quartically in $\{(x^{(i,j)}, u^{(i,j)})\}_{j=1}^{N_1}$ and $(x^{(i)}_{N_0}, u^{(i)}_{N_0})$, with $\mO(1)$ coefficients. Note that $\{(x^{(i,j)}, u^{(i,j)})\}_{j=1}^{N_1}$ and $(x^{(i)}_{N_0}, u^{(i)}_{N_0})$ are normally distributed with $0$ mean and $\mO(1)$ covariance matrix. Therefore,
\begin{equation*}
\begin{aligned}
    & \quad \EE \left[ \left\| \widehat{\nabla L_t}(\theta_t) \right\|^2_F ~\Big|~ \mG_t \right] = \EE \left[ \left\| \frac1N \sum_{i=1}^N \widehat{f}(x^{(i)}_{N_0}, u^{(i)}_{N_0}) \right\|^2_F ~\Big|~ \mG_t \right] = \mO(1)
\end{aligned}
\end{equation*}
So \eqref{eq:critic_grad_bound} holds with $c_L = \mO(1)$. We also see that $c_L = \text{poly}(d+k)$ as the dimensions increase. We will show \eqref{eq:error_grad} next.
By definition,
\begin{equation}\label{eq:bias}
    \left\| \EE \left[ \widehat{\nabla L_t}(\theta_t) - \nabla L_t(\theta_t) ~\Big|~ \mG_t \right] \right\|_F =  \| \EE_{N_0} [f(x,u)] - \EE_{K_t} [f(x,u)] \|_F.
\end{equation}
Here, we remind the reader that the expectation on the left in \eqref{eq:bias} is taken w.r.t. the training filtration $\mG_t$ while those on the right are taken w.r.t. the state and action distributions.

We remark that existing results \citep{arnold1968ergodic} bound \eqref{eq:bias} directly. However, it can be computed directly, so we give an elementary proof. Recall that the state trajectory is given by 
\begin{equation*}
    x_{s+1} =  (A -BK_t) x_s + \ep_s
\end{equation*}
with $x_0 = 0$ where $\ep_s \sim N(0, D_{\ep})$. Therefore, the distribution of $x_{N_0}$ is
$$x_{N_0} \sim N\left(0, \sum_{s=0}^{N_0-1} (A -BK_t)^s D_{\ep} ((A -BK_t)\tp)^s \right) =: N\left(0, D^{(N_0)}_{K_t}\right)$$
and the stationary distribution of $x_s$ is
$$x_{\infty} \sim N\left(0, \sum_{s=0}^{\infty} (A -BK_t)^s D_{\ep} ((A -BK_t)\tp)^s \right) = N\left(0, D_{K_t}\right).$$
Since $\rho(A-BK_t) \le \rho < 1$, $D_{\ep} > 0$, and $N_0 = \mO(\log \frac{1}{\delta})$, we have $D_{K_t} > \sigma_{min}(D_{\ep})$, $D^{(N_0)}_{K_t} > \sigma_{min}(D_{\ep})$, $D_{K_t} - D^{(N_0)}_{K_t} \ge 0$ and $\| D_{K_t} - D^{(N_0)}_{K_t} \|_F \lesssim \delta$.
Since $u_s \sim N(-K_t x_s, \sigma^2 I_k)$, we have the joint distribution for $z_{N_0} = (x_{N_0}\tp, u_{N_0}\tp)\tp$
$$z_{N_0} \sim N\left( 0, \begin{bmatrix}  D^{(N_0)}_{K_t} & - D^{(N_0)}_{K_t} K_t\tp \\ - K_t D^{(N_0)}_{K_t} & K_t  D^{(N_0)}_{K_t} K_t\tp + \sigma^2 I_k \end{bmatrix} \right) =: N\left( 0, \Sigma^{(N_0)}_{K_t} \right)$$
and the joint stationary distribution
$$z \sim N\left( 0, \begin{bmatrix}  D_{K_t} & - D_{K_t} K_t\tp \\ - K_t D_{K_t} & K_t  D_{K_t} K_t\tp + \sigma^2 I_k \end{bmatrix} \right) =: N\left( 0, \Sigma_{K_t} \right)$$
Since $\| D_{K_t} - D^{(N_0)}_{K_t} \|_F \lesssim \delta$ and $\|K_t\| \le c_K$, we have $\| \Sigma_{K_t} - \Sigma^{(N_0)}_{K_t} \|_F \le c_6 \delta$. 
Here the positive constant $c_6 = \mO(1)$ decrease geometrically as $N_0$ increases algebraically.
Furthermore, using the same argument when we prove $\Sigma_K \ge \mu_{\Sigma}$ in Lemma \ref{lem:critic_convex}, we can find a positive constant $\mu_{\Sigma} = \mO(1)$ such that $\Sigma_{K_t} \ge \mu_{\Sigma}$ and $\Sigma^{(N_0)}_{K_t} \ge \mu_{\Sigma}$. Therefore
\begin{equation}\label{eq:diff_exp}
\begin{aligned}
     & \quad \left\| \EE_{N_0} [f(x,u)] - \EE_{K_t} [f(x,u)] \right\|_F \\
     & = \left\| \int_{\RR^{d+k}} f(z) (2\pi)^{-\frac{d+k}{2}} \left[ \det(\Sigma^{(N_0)}_{K_t})^{-\frac12} \exp\left( -\frac12 z\tp (\Sigma^{(N_0)}_{K_t})^{-1} z \right) \right.\right.\\
     & \left.\left. \hspace{1in} - \det(\Sigma_{K_t})^{-\frac12} \exp\left( -\frac12 z\tp (\Sigma_{K_t})^{-1} z \right) \right] \rd z \right\|_F \\
     & \le \int_{\RR^{d+k}} c(1 + |z|^4) \Bigg| \det(\Sigma^{(N_0)}_{K_t})^{-\frac12} \exp\left( -\frac12 z\tp (\Sigma^{(N_0)}_{K_t})^{-1} z \right) \\
     &  \hspace{1in} - \det( \Sigma_{K_t})^{-\frac12} \exp\left( -\frac12 z\tp (\Sigma_{K_t})^{-1} z \right) \Bigg| \rd z \\
     & \le \int_{\RR^{d+k}} c(1 + |z|^4) \left[ \det(\Sigma^{(N_0)}_{K_t})^{-\frac12} - \det(\Sigma_{K_t})^{-\frac12} \right] \exp\left( -\frac12 z\tp (\Sigma^{(N_0)}_{K_t})^{-1} z \right) \rd z \\
     & \quad + \int_{\RR^{d+k}} c(1 + |z|^4) \det(\Sigma_{K_t})^{-\frac12} \left[ \exp\left( -\frac12 z\tp (\Sigma_{K_t})^{-1} z \right) - \exp\left( -\frac12 z\tp (\Sigma^{(N_0)}_{K_t})^{-1} z \right) \right] \rd z 
\end{aligned}
\end{equation}
There is no absolute value at the end of \eqref{eq:diff_exp} because each term is non-negative. Next, we will bound the two integrals respectively. For the first one, we have
\begin{equation*}
\begin{aligned}
     & \quad \det(\Sigma^{(N_0)}_{K_t})^{-\frac12} - \det(\Sigma_{K_t})^{-\frac12}\\
     & = \dfrac{\det(\Sigma_{K_t}) - \det(\Sigma^{(N_0)}_{K_t})}{ \sqrt{\det(\Sigma^{(N_0)}_{K_t})\det(\Sigma_{K_t})} \left(\sqrt{\det(\Sigma_{K_t})} + \sqrt{\det(\Sigma^{(N_0)}_{K_t})} \right) }\\
     & = \mO(1) \left( \det(\Sigma_{K_t}) - \det(\Sigma^{(N_0)}_{K_t}) \right).
\end{aligned}
\end{equation*}
Next, we will show $\det(\Sigma_{K_t}) - \det(\Sigma^{(N_0)}_{K_t}) = \mO(\delta)$. We can find a unitary matrix $U$ such that $U\tp \Sigma^{(N_0)}_{K_t} U$ is a diagonal matrix, 
$$\left\| U\tp \Sigma_{K_t} U - U\tp \Sigma^{(N_0)}_{K_t} U \right\|_F = \left\| \Sigma_{K_t} - \Sigma^{(N_0)}_{K_t} \right\|_F \le c_6 \delta,$$
and
$$\det(\Sigma_{K_t}) - \det(\Sigma^{(N_0)}_{K_t}) = \det(U\Sigma_{K_t}U\tp) - \det(U\Sigma^{(N_0)}_{K_t}U\tp). $$
If we assume that the diagonal element of $U\Sigma_{K_t}U\tp$ to be $a_1, \cdots, a_{d+k}$ and
$$U\Sigma^{(N_0)}_{K_t}U\tp = \text{diag} (b_1, \cdots, b_{d+k}).$$
Then $a_i \ge b_i$ and $a_i - b_i = \mO(\delta)$. Therefore
$$0 \le \det(U\Sigma_{K_t}U\tp) - \det(U\Sigma^{(N_0)}_{K_t}U\tp) \le \prod_{i=1}^{d+k} a_i - \prod_{i=1}^{d+k} b_i = \mO(\delta).$$
Therefore, $\det(\Sigma_{K_t}) - \det(\Sigma^{(N_0)}_{K_t}) = \mO(\delta)$ and hence
$$\det(\Sigma^{(N_0)}_{K_t})^{-\frac12} - \det(\Sigma_{K_t})^{-\frac12} \le c \delta$$
with positive constant $c$ being as small as we want (through increasing $N_0$). 
Therefore, the first integral in \eqref{eq:diff_exp} satisfies
\begin{equation}\label{eq:diff_exp1}
\begin{aligned}
     & \quad \int_{\RR^{d+k}} c(1 + |z|^4) \left[ \det(\Sigma^{(N_0)}_{K_t})^{-\frac12} - \det(\Sigma_{K_t})^{-\frac12} \right] \exp\left( -\frac12 z\tp (\Sigma^{(N_0)}_{K_t})^{-1} z \right) \rd z \\
    & \le c \delta \int_{\RR^{d+k}} (1 + |z|^4) \exp\left( -\frac12 z\tp (\Sigma^{(N_0)}_{K_t})^{-1} z \right) \rd z = c \delta \mO(1) \le \frac{1}{2} \delta.
\end{aligned}
\end{equation}
Here, again, the constant $c$ may differ according to the context. 
A more detailed computation shows that
$$\det(\Sigma^{(N_0)}_{K_t})^{-\frac12} - \det(\Sigma_{K_t})^{-\frac12} \le \det(\Sigma^{(N_0)}_{K_t})^{-\frac12} ~\text{poly}(d+k)~c_6 \delta.$$
Therefore, $N_0$ should scale with $\log(d+k)$ as the dimensions increase.
Next, we bound the second integration in \eqref{eq:diff_exp}. Using the inequality $1 - e^{-x} \le x$, we have
\begin{equation*}
\begin{aligned}
     & \quad \exp\left( -\frac12 z\tp (\Sigma_{K_t})^{-1} z \right) - \exp\left( -\frac12 z\tp (\Sigma^{(N_0)}_{K_t})^{-1} z \right) \\
     & = \exp\left( -\frac12 z\tp (\Sigma_{K_t})^{-1} z \right) \left[ 1 - \exp\left( -\frac12 z\tp \left((\Sigma^{(N_0)}_{K_t})^{-1} - (\Sigma_{K_t})^{-1}\right) z \right) \right] \\
     & \le \frac12 z\tp \left((\Sigma^{(N_0)}_{K_t})^{-1} - (\Sigma_{K_t})^{-1}\right) z \exp\left( -\frac12 z\tp (\Sigma_{K_t})^{-1} z \right) \\
     & =  \frac12 \exp\left( -\frac12 z\tp (\Sigma_{K_t})^{-1} z \right) \Tr\left[ \left((\Sigma^{(N_0)}_{K_t})^{-1} - (\Sigma_{K_t})^{-1}\right) z z\tp \right] \\
     & = \frac12 \exp\left( -\frac12 z\tp (\Sigma_{K_t})^{-1} z \right) \Tr\left[ (\Sigma^{(N_0)}_{K_t})^{-1} \left(\Sigma_{K_t} - \Sigma^{(N_0)}_{K_t} \right) (\Sigma_{K_t})^{-1}  z z\tp \right] \\
     & \le \frac12 \exp\left( -\frac12 z\tp (\Sigma_{K_t})^{-1} z \right) \|(\Sigma^{(N_0)}_{K_t})^{-1} \left(\Sigma_{K_t} - \Sigma^{(N_0)}_{K_t} \right) (\Sigma_{K_t})^{-1}\| \, \Tr[z z\tp]\\
     & \le \frac12 \exp\left( -\frac12 z\tp (\Sigma_{K_t})^{-1} z \right) \dfrac{1}{\mu_{\Sigma}^2} c_6 \delta |z|^2.
\end{aligned}
\end{equation*}
Therefore, the second integration in \eqref{eq:diff_exp} satisfies
\begin{equation}\label{eq:diff_exp2}
\begin{aligned}
     & \quad \int_{\RR^{d+k}} c(1 + |z|^4) \det(\Sigma_{K_t})^{-\frac12} \left[ \exp\left( -\frac12 z\tp (\Sigma_{K_t})^{-1} z \right) - \exp\left( -\frac12 z\tp (\Sigma^{(N_0)}_{K_t})^{-1} z \right) \right] \rd z \\
     & \le \delta \int_{\RR^{d+k}} c(|z|^2 + |z|^6) \det(\Sigma_{K_t})^{-\frac12} \exp\left( -\frac12 z\tp (\Sigma_{K_t})^{-1} z \right) \rd z = \delta c \mO(1) \le \frac12 \delta.
\end{aligned}
\end{equation}
Plugging \eqref{eq:diff_exp1} and  \eqref{eq:diff_exp2} into \eqref{eq:diff_exp}, we obtain
\begin{equation*}
    \left\| \EE_{N_0} [f(x,u)] - \EE_{K_t} [f(x,u)] \right\|_F \le \delta.
\end{equation*}
\end{proof}

\begin{proof}[Proof of Lemma \ref{lem:thetaK}]
By definition
\begin{equation*}
    \theta_K - \theta_{K'} = \begin{bmatrix}  A\tp (P_K - P_{K'}) A & A\tp (P_K - P_{K'}) B \\ B\tp (P_K - P_{K'}) A & B\tp (P_K - P_{K'}) B \end{bmatrix} = \begin{bmatrix} A\tp \\ B\tp \end{bmatrix} \begin{bmatrix} P_K - P_{K'} \end{bmatrix} \begin{bmatrix} A & B \end{bmatrix}
\end{equation*}
Therefore,
\begin{equation}\label{eq:thetaP}
\begin{aligned}
     & \quad \| \theta_K - \theta_{K'} \|_F^2 = \Tr[(\theta_K - \theta_{K'})\tp (\theta_K - \theta_{K'})] \\
     & = \Tr\left( [(AA\tp + BB\tp) (P_K - P_{K'})]^2   \right) \le (\|A\|^2 + \|B\|^2)^2 \|P_K - P_{K'}\|_F^2
\end{aligned}
\end{equation}
Therefore, our goal is to bound $\|P_K - P_{K'}\|_F$ by $\|K - K'\|_F$. By definition in \eqref{eq:PK},
\begin{equation*}
\begin{aligned}
     & \quad P_K - P_{K'} \\
     & = K\tp R K - K\ptp R K' + (A-BK)\tp P_K (A-BK) - (A-BK')\tp P_{K'} (A-BK') \\
     & = K\tp R K - K\tp R K' + K\tp R K' - K\ptp R K' \\
     & \quad + (A-BK)\tp P_K (A-BK) - (A-BK)\tp P_K (A-BK')\\
     & \quad +  (A-BK)\tp P_K (A-BK') - (A-BK)\tp P_{K'} (A-BK')\\
     & \quad + (A-BK)\tp P_{K'} (A-BK') - (A-BK')\tp P_{K'} (A-BK')\\
     & = K\tp R (K - K') + (K - K')\tp R K' - (A-BK)\tp P_K B (K - K') \\
     & \quad + (A-BK)\tp (P_K - P_{K'}) (A-BK') - (K - K')\tp B\tp P_{K'} (A-BK')
\end{aligned}
\end{equation*}
Therefore, 
\begin{equation}\label{eq:PK_diff}
\begin{aligned}
     & \quad P_K - P_{K'} - (A-BK)\tp (P_K - P_{K'}) (A-BK') \\
     & = K\tp R (K - K') + (K - K')\tp R K' \\
     & \quad - (A-BK)\tp P_K B (K - K') - (K - K')\tp B\tp P_{K'} (A-BK')
\end{aligned}
\end{equation}
Next, we want to take $\|\cdot\|_F$ on both sides of \eqref{eq:PK_diff}. For the left hand side, since $\rho(A-BK), \rho(A-BK') \le \rho < 1$ and $\|A-BK\|, \|A-BK'\| \le c_A$, we can repeat the last part in the proof of Lemma \ref{lem:critic_convex} and prove that
\begin{equation}\label{eq:PK_diff_LHS}
\|P_K - P_{K'}\|_F \le c \|(P_K - P_{K'}) - (A-BK)\tp (P_K - P_{K'}) (A-BK') \|_F
\end{equation}
where $c$ is proportional to $1 / (1 - \rho^2)$ and also depends on $c_A$ and $d$.
For the right hand side of \eqref{eq:PK_diff}, since $\|P_K\| \le c_P$, $\|P_{K'}\| \le c_P$, $\|K\| \le c_K$ and $\|K'\| \le c_K$,
\begin{equation}\label{eq:PK_diff_RHS}
\begin{aligned}
     & \quad \|K\tp R (K - K') + (K - K')\tp R K' \\
     & \quad - (A-BK)\tp P_K B (K - K') - (K - K')\tp B\tp P_{K'} (A-BK')\|_F \\
     & \le 2 (c_K \|R\| + c_P c_A \|B\|) \, \| K - K' \|_F.
\end{aligned}
\end{equation}
Plugging \eqref{eq:PK_diff_LHS} and \eqref{eq:PK_diff_RHS} into \eqref{eq:PK_diff}, we obtain
\begin{equation}\label{eq:PK_diff2}
    \|P_K - P_{K'}\|_F \le 2c (c_K \|R\| + c_P c_A \|B\|) \, \| K - K' \|_F.
\end{equation}
Finally, combining \eqref{eq:thetaP} and \eqref{eq:PK_diff2}, we obtain
\begin{equation}
    \| \theta_K - \theta_{K'} \|_F \le c_1 \|K-K'\|_F
\end{equation}
with $c_1 = 2c (c_K \|R\| + c_P c_A \|B\|) \, (\|A\|^2 + \|B\|^2)$. This $c_1$ grows polynomially as the dimensions increase.
\end{proof}

\begin{proof}[Proof of Lemma \ref{lem:critic_improvement}]
Note that
\begin{equation}\label{eq:long}
\begin{aligned}
    & \quad \| \theta_{t+1} - \theta_{K_{t+1}} \|^2_F = \| \theta_t - \alpha_t \widehat{\nabla L_t}(\theta_t) - \theta_{K_t} + \theta_{K_t} - \theta_{K_{t+1}} \|^2_F \\
    & = \| \theta_t - \theta_{K_t} \|^2_F - 2 \alpha_t \Tr\left[ (\theta_t - \theta_{K_t})\tp \widehat{\nabla L_t}(\theta_t) \right] \\
    & \quad + \alpha_t^2 \| \widehat{\nabla L_t}(\theta_t) \|^2_F + \| \theta_{K_t} - \theta_{K_{t+1}} \|^2_F + 2\Tr\left[ (\theta_{K_t} - \theta_{K_{t+1}})\tp (\theta_t - \theta_{K_t} - \alpha_t \widehat{\nabla L_t}(\theta_t)) \right] \\
    & = \| \theta_t - \theta_{K_t} \|^2_F - 2 \alpha_t \Tr\left[ (\theta_t - \theta_{K_t})\tp \nabla L_t(\theta_t) \right] + 2 \alpha_t \Tr\left[ (\theta_t - \theta_{K_t})\tp (\nabla L_t(\theta_t) - \widehat{\nabla L_t}(\theta_t))) \right] \\
    &\quad + \alpha_t^2 \| \widehat{\nabla L_t}(\theta_t) \|^2_F + \| \theta_{K_t} - \theta_{K_{t+1}} \|^2_F + 2\Tr\left[ (\theta_{K_t} - \theta_{K_{t+1}})\tp (\theta_t - \theta_{K_t} - \alpha_t \widehat{\nabla L_t}(\theta_t)) \right] \\
    & \le (1 - 2\alpha_t \mu_{\sigma}) \| \theta_t - \theta_{K_t} \|^2_F + 2 \alpha_t \Tr\left[ (\theta_t - \theta_{K_t})\tp (\nabla L_t(\theta_t) - \widehat{\nabla L_t}(\theta_t))) \right] + \alpha_t^2 \| \widehat{\nabla L_t}(\theta_t) \|^2_F \\
    & \quad + \| \theta_{K_t} - \theta_{K_{t+1}} \|^2_F + 2\Tr\left[ (\theta_{K_t} - \theta_{K_{t+1}})\tp (\theta_t - \theta_{K_t}) \right] - 2\alpha_t  \Tr\left[ (\theta_{K_t} - \theta_{K_{t+1}})\tp \widehat{\nabla L_t}(\theta_t) \right]\\
    & \le (1 - \frac53 \alpha_t \mu_{\sigma}) \| \theta_t - \theta_{K_t} \|^2_F + 2 \alpha_t \Tr\left[ (\theta_t - \theta_{K_t})\tp (\nabla L_t(\theta_t) - \widehat{\nabla L_t}(\theta_t))) \right] + 2\alpha_t^2 \| \widehat{\nabla L_t}(\theta_t) \|^2_F \\
    & \quad + (\dfrac{3}{\alpha_t \mu_{\sigma}} + 2) \| \theta_{K_t} - \theta_{K_{t+1}} \|^2_F 
\end{aligned}
\end{equation}
The first inequality is because $L_t(\theta)$ is $\mu_{\sigma}-$ strongly convex and hence
$$\Tr\left[ (\theta_t - \theta_{K_t})\tp \nabla L_t(\theta_t) \right] = \Tr\left[ (\theta_t - \theta_{K_t})\tp (\nabla L_t(\theta_t) - \nabla L_t(\theta_{K_t})) \right] \ge \mu_{\sigma} \| \theta_t - \theta_{K_t} \|_F^2.$$
The second inequality in \eqref{eq:long} is a simple application of Cauchy-Schwartz inequality. Taking expectation w.r.t. $\mG_t$ in \eqref{eq:long}, we obtain
\begin{equation*}
\begin{aligned}
     & \quad \EE \left[ \| \theta_{t+1} - \theta_{K_{t+1}} \|^2_F ~\big|~ \mG_t \right]\\
     & \le (1 - \frac53 \alpha_t \mu_{\sigma}) \| \theta_t - \theta_{K_t} \|^2_F + 2 \alpha_t \Tr\left[ (\theta_t - \theta_{K_t})\tp \EE\left[\nabla L_t(\theta_t) - \widehat{\nabla L_t}(\theta_t) ~\Big|~ \mG_t \right] \right] \\
    & \quad + 2\alpha_t^2 \EE\left[\left\| \widehat{\nabla L_t}(\theta_t) \right\|^2_F ~\Big|~ \mG_t \right] + (\dfrac{3}{\alpha_t \mu_{\sigma}} + 2) \| \theta_{K_t} - \theta_{K_{t+1}} \|^2_F\\
    & \le (1 - \frac43 \alpha_t \mu_{\sigma}) \| \theta_t - \theta_{K_t} \|^2_F + \dfrac{3 \alpha_t}{\mu_{\sigma}} \left\| \EE\left[\nabla L_t(\theta_t) - \widehat{\nabla L_t}(\theta_t) ~\Big|~ \mG_t \right] \right\|_F^2  \\
    & \quad + 2\alpha_t^2 \EE\left[\left\| \widehat{\nabla L_t}(\theta_t) \right\|^2_F ~\Big|~ \mG_t \right] + (\dfrac{3}{\alpha_t \mu_{\sigma}} + 2) \| \theta_{K_t} - \theta_{K_{t+1}} \|^2_F.
\end{aligned}
\end{equation*}
Therefore,
\begin{equation*}
\begin{aligned}
    & \quad \EE \left[ \| \theta_{t+1} - \theta_{K_{t+1}} \|^2_F ~\big|~ \mG_t \right] - \| \theta_t - \theta_{K_t} \|_F^2 \\
    & \le  - \frac43 \alpha_t \mu_{\sigma} \| \theta_t - \theta_{K_t} \|^2_F + \dfrac{3 \alpha_t}{\mu_{\sigma}} \left\| \EE\left[\nabla L_t(\theta_t) - \widehat{\nabla L_t}(\theta_t) ~\Big|~ \mG_t \right] \right\|_F^2\\
    & \quad + 2\alpha_t^2 \EE\left[\left\| \widehat{\nabla L_t}(\theta_t) \right\|^2_F ~\Big|~ \mG_t \right] + (\dfrac{3}{\alpha_t \mu_{\sigma}} + 2) \| \theta_{K_t} - \theta_{K_{t+1}} \|^2_F.
\end{aligned}
\end{equation*}
Combining with \eqref{eq:error_grad}, \eqref{eq:critic_grad_bound}, and the definition of $\alpha_t$, we obtain \eqref{eq:critic_imporvement0}:
\begin{equation*}
\begin{aligned}
     & \quad \EE \left[ \| \theta_{t+1} - \theta_{K_{t+1}} \|^2_F ~\big|~ \mG_t \right] - \| \theta_t - \theta_{K_t} \|_F^2 \\
     & \le - \frac43 \alpha_t \mu_{\sigma} \| \theta_t - \theta_{K_t} \|_F^2 + \frac14 \dfrac{\sigma_{min}(D_{\ep})}{c_3} \beta_t \ve + \bigl(\dfrac{3}{\alpha_t \mu_{\sigma}} + 2\bigr) \| \theta_{K_t} - \theta_{K_{t+1}} \|_F^2.
\end{aligned}
\end{equation*}
\end{proof}

\subsubsection{Proofs for the Actor}
Next, we prove the results for the actor.

\begin{proof}[Proof of Lemma \ref{lem:gradient_dominant}]
We prove the upper bound first. According to \eqref{eq:JK},
\begin{equation}\label{eq:JK_dif}
    J(K) - J(K^*) = \Tr((P_K - P_{K^*}) D_{\ep}) = \EE_{x \sim N(0, D_{\ep})} [x\tp (P_K - P_{K^*}) x]
\end{equation}
where we recall that $P_K = (Q + K\tp R K) + (A-BK)\tp P_K (A-BK)$ and $P_{K^*}$ satisfies a similar equation. So, $P_{K^*}$ also has the following expression in series
\begin{equation*}
    P_{K^*} = \sum_{s=0}^{\infty} [(A-BK^*)^s]\tp (Q + K^{* \top} R K^*) (A-BK^*)^s.
\end{equation*}
Therefore, if we define a sequence $\{y_s\}_{s=0}^{\infty}$ with $y_0 = x$ and $y_{s+1} = (A-BK^*) y_s$, then
\begin{equation*}
    x\tp P_{K^*} x = \sum_{s=0}^{\infty} x\tp [(A-BK^*)^s]\tp (Q + K^{* \top} R ^*K) (A-BK^*)^s x = \sum_{s=0}^{\infty} y_s\tp (Q + K^{* \top} R K^*) y_s.
\end{equation*}
Combining with 
$$x\tp P_{K} x = \sum_{s=0}^{\infty} \left( y_s\tp P_K y_s - y_{s+1}\tp P_K y_{s+1} \right) = \sum_{s=0}^{\infty} y_s\tp (P_K - (A-BK^*)\tp P_K (A-BK^*)) y_s$$
and \eqref{eq:JK_dif}, we obtain
\begin{equation}\label{eq:JK_dif2}
\begin{aligned}
    & \quad J(K) - J(K^*) \\
    &= \EE_{D_{\ep}, K^*} \left[ \sum_{s=0}^{\infty} y_s\tp \left(-Q - K^{* \top} R K^* + P_K - (A-BK^*)\tp P_K (A-BK^*)\right) y_s \right]\\
    & = \Tr\left[ \EE_{D_{\ep}, K^*} \left[ \sum_{s=0}^{\infty} y_s  y_s\tp \right] \cdot \left(-Q - K^{* \top} R K^* + P_K - (A-BK^*)\tp P_K (A-BK^*)\right) \right]
\end{aligned}
\end{equation}
where $\EE_{D_{\ep}, K^*}$ denotes the expectation with $y_0 \sim N(0, D_{\ep})$ and $y_{s+1} = (A-BK^*) y_s$. Next, we analyze the two terms in \eqref{eq:JK_dif2} respectively. The first term is easy, recall that $D_{K^*}$ is the solution of 
$$D_{K^*} = D_{\ep} + (A-BK^*) D_{K^*} (A-BK^*)\tp$$
so that
$$D_{K^*} = \sum_{s=0}^{\infty} (A-BK^*)^s D_{\ep} [(A-BK^*)\tp]^s.$$ 
Therefore,
\begin{equation}\label{eq:DK_star}
    \EE_{D_{\ep}, K^*} \left[ \sum_{s=0}^{\infty} y_s  y_s\tp \right] = \EE_{x \sim N(0, D_{\ep})} \left[ \sum_{s=0}^{\infty} (A-BK^*)^s x x\tp [(A-BK^*)\tp]^s \right] = D_{K^*}.
\end{equation}
Next, we consider the second term in \eqref{eq:JK_dif2}. By direct computation,
\begin{equation}\label{eq:JK_dif3}
\begin{aligned}
    & \quad -Q - K^{* \top} R K^* + P_K - (A-BK^*)\tp P_K (A-BK^*)\\
    & = -Q - (K^{*}-K+K)\tp R (K^{*}-K+K) + P_K \\
    & \quad- (A -BK+BK-BK^*)\tp P_K (A-BK+BK-BK^*)\\
    & = (K-K^*)\tp (RK - B\tp P_K(A-BK)) + (RK - B\tp P_K(A-BK))\tp (K-K^*)\\
    & \quad - (K-K^*)\tp (R + B\tp P_K B) (K-K^*)\\
    & = (K-K^*)\tp G_K + G_K\tp (K-K^*) - (K-K^*)\tp (R + B\tp P_K B) (K-K^*)\\
    & = G_K\tp (R + B\tp P_K B)^{-1} G_K \\
    & \quad - (K-K^*-(R + B\tp P_K B)^{-1}G_K)\tp (R + B\tp P_K B) (K-K^*-(R + B\tp P_K B)^{-1}G_K)\\
    & \le G_K\tp (R + B\tp P_K B)^{-1} G_K
\end{aligned}
\end{equation}
where we have used the equation \eqref{eq:PK} for $P_K$ in the second equality and the definition of $G_K$ \eqref{eq:GK} in the third equality. The $\le$ above means the difference of the two matrix is positive semi-definite. Plugging \eqref{eq:DK_star} and \eqref{eq:JK_dif3} into \eqref{eq:JK_dif2}, we obtain
\begin{equation*}
    J(K) - J(K^*) \le \Tr(D_{K^*} \, G_K\tp (R + B\tp P_K B)^{-1} G_K) \le \|D_{K^*}\| / \sigma_{min}(R) \Tr(G_K G_K\tp).
\end{equation*}
This finishes the proof of the upper bound. Next, we prove the lower bound. Note that the argument above does not rely on the optimality of $K^*$. Therefore, we can obtain a general formula (that is useful in the proof later):
\begin{multline}\label{eq:JK_diff_gen}
    J(K) - J(K') \\
    = \Tr\left[ D_{K'} \left( (K - K')\tp G_K + G_{K}\tp (K - K') - (K - K')\tp (R + B\tp P_{K} B) (K - K') \right) \right].
\end{multline}
Specifically, we can set $K' = K-(R + B\tp P_K B)^{-1}G_K$ (i.e., let \eqref{eq:JK_dif3} hold with equality), then by the optimality of $K^*$ and \eqref{eq:JK_diff_gen}, we obtain
\begin{equation*}
\begin{aligned}
    & \quad J(K) - J(K^*) \ge J(K) - J(K') = \Tr(D_{K'} \, G_K\tp (R + B\tp P_K B)^{-1} G_K)\\
    & \ge \sigma_{min}(D_{\ep}) \, \| R + B\tp P_K B \|^{-1} \Tr(G_K G_K\tp) \ge \dfrac{\sigma_{min}(D_{\ep})}{\|R\| + c_P \|B\|^2} \Tr(G_K G_K\tp)
\end{aligned}
\end{equation*}
\end{proof}

\begin{proof}[Proof of Lemma \ref{lem:actor_improvement}]
By \eqref{eq:JK_diff_gen},
\begin{equation*}
\begin{aligned}
    & \quad J(K_{t}) - J(K_{t+1}) \\
    & = \Tr\left[ D_{K_{t+1}} \left( (K_t - K_{t+1})\tp G_{K_t} + G_{K_t}\tp (K_t - K_{t+1}) \right. \right.\\
    & \quad \left. \left. - (K_t - K_{t+1})\tp (R + B\tp P_{K_t} B) (K_t - K_{t+1}) \right) \right]\\
    & = \Tr\left[ D_{K_{t+1}} \left( \beta_t \widehat{G}_{K_t}\tp G_{K_t} + \beta_t G_{K_t}\tp \widehat{G}_{K_t} - \beta_t^2 \widehat{G}_{K_t}\tp (R + B\tp P_{K_t} B) \widehat{G}_{K_t} \right) \right]
\end{aligned}
\end{equation*}
Therefore,
\begin{equation*}
\begin{aligned}
    & \quad J(K_{t+1}) - J(K_t) \\
    & = - \beta_t \Tr\left[ D_{K_t} \left( \widehat{G}_{K_t}\tp G_{K_t} + G_{K_t}\tp \widehat{G}_{K_t} - \beta_t \widehat{G}_{K_t}\tp (R + B\tp P_{K_t} B) \widehat{G}_{K_t} \right) \right] \\
    & = -\beta_t \Tr\left[ D_{K_t} \left( G_{K_t}\tp G_{K_t} + \widehat{G}_{K_t}\tp \widehat{G}_{K_t} - ( G_{K_t} - \widehat{G}_{K_t})\tp ( G_{K_t} - \widehat{G}_{K_t}) - \beta_t \widehat{G}_{K_t}\tp (R + B\tp P_{K_t} B) \widehat{G}_{K_t} \right) \right]
\end{aligned}
\end{equation*}
Recall that we proved
$$\sigma_{min}(D_{\ep}) I_{d} \le D_{K_t} \le c_D I_{d} ~~ \text{and} ~~ P_{K_t} \le c_P$$
in Lemma \ref{lem:bound}. Therefore, 
$$\Tr\left[ D_{K_t} G_{K_t}\tp G_{K_t} \right] \ge \sigma_{min}(D_{\ep}) \| G_{K_t} \|^2_F,$$
$$\Tr\left[ D_{K_t} \widehat{G}_{K_t}\tp \widehat{G}_{K_t} \right] \ge \sigma_{min}(D_{\ep}) \| \widehat{G}_{K_t} \|^2_F,$$
$$\Tr\left[ D_{K_t} \widehat{G}_{K_t}\tp (R + B\tp P_{K_t} B) \widehat{G}_{K_t} \right] \le c_D(\|R\| + c_P \|B\|^2) \| \widehat{G}_{K_t} \|^2_F,$$
and
$$\Tr\left[ D_{K_t} ( G_{K_t} - \widehat{G}_{K_t})\tp ( G_{K_t} - \widehat{G}_{K_t}) \right] \le c_D \| G_{K_t} - \widehat{G}_{K_t} \|^2_F.$$
Therefore,
\begin{align*}
    J(K_{t+1}) - J(K_t) & \le -\beta_t \sigma_{min}(D_{\ep}) (\|G_{K_t}\|_F^2 + \|\widehat{G}_{K_t}\|_F^2) + \beta_t c_D \|G_{K_t} - \widehat{G}_{K_t}\|_F^2\\
    & \quad + \beta_t^2 c_D(\|R\| + c_P \|B\|^2) \| \widehat{G}_{K_t} \|^2_F
\end{align*}
Finally, by Lemma \ref{lem:gradient_dominant}, we can conclude that 
\begin{align*}
    & J(K_{t+1}) - J(K_t) \le -\beta_t \dfrac{\sigma_{min}(D_{\ep})}{c_3} (J(K_t) - J(K^*))\\
    & \quad - \beta_t  \left[\sigma_{min}(D_{\ep}) - \beta_t c_D (\|R\| + c_P \|B\|^2) \right] \|\widehat{G}_{K_t}\|_F^2 + \beta_t c_D \|G_{K_t} - \widehat{G}_{K_t}\|_F^2
\end{align*}
\end{proof}
\subsubsection{Proofs for the main theorem}
Finally we can prove our main theorem.
\begin{proof}[Proof of Theorem \ref{thm:main1}]
By lemma \ref{lem:critic_gradient_accuracy}, \eqref{eq:error_grad} and \eqref{eq:critic_grad_bound} hold for all $t \le T$. We define a Lyapunov function
\begin{equation*}
    \mL_t = \mL(\theta_t, K_t) = \|\theta_t - \theta_{K_t}\|_F^2 + J(K_t) - J(K^*).
\end{equation*}
Firstly, $\mL_0 = \mO(1)$ because
$$ \|\theta_0 - \theta_{K_0}\|_F^2 = \| \theta_{K_0} \|_F^2 = \left\| \begin{bmatrix} Q + A\tp P_{K_0} A & A\tp P_{K_0} B \\ B\tp P_{K_0} A & R + B\tp P_{K_0} B \end{bmatrix} \right\|_F^2 = \mO(1) $$
(note that $P_{K_0} = Q + A\tp P_{K_0} A$ implies $\|P_{K_0}\|_F = \mO(1)$) and
$$J(K_0) - J(K^*) \le J(K_0) = \Tr(D_{\ep} P_{K_0}) + \sigma^2 \Tr(R) \le c_P \Tr[D_{\ep}] + \sigma^2 \Tr(R) = \mO(1).$$
Next, we want to show a decrease rate of the Lyapunov function. According to Lemma \ref{lem:critic_improvement} and Lemma \ref{lem:actor_improvement},
\begin{equation}\label{eq:Lyap_diff}
\begin{aligned}
    & \quad \EE \left[ \mL_{t+1} ~|~ \mG_t \right] - \mL_t  \\
    & \le - \frac43 \alpha_t \mu_{\sigma} \| \theta_t - \theta_{K_t} \|_F^2 + \frac14 \dfrac{\sigma_{min}(D_{\ep})}{c_3} \beta_t \ve + (\dfrac{3}{\alpha_t \mu_{\sigma}} + 2) \| \theta_{K_t} - \theta_{K_{t+1}} \|_F^2 \\
    & \quad -\beta_t \dfrac{\sigma_{min}(D_{\ep})}{c_3} (J(K_t) - J(K^*)) - \beta_t \left[\sigma_{min}(D_{\ep}) - \beta_t c_D (\|R\| + c_P \|B\|^2) \right] \|\widehat{G}_{K_t}\|_F^2 \\
    & \quad + \beta_t c_D \|G_{K_t} - \widehat{G}_{K_t}\|_F^2.
\end{aligned}
\end{equation}
Fortunately, we can use the negative term in the actor estimate to bound the positive term in the critic estimate and use the negative term in the critic estimate to bound the positive term in the actor estimate. Specifically, by Lemma \ref{lem:thetaK},
$$\| \theta_{K_t} - \theta_{K_{t+1}} \|^2_F \le c_1^2 \| K_t - K_{t+1} \|_F^2 = c_1^2 \beta_t^2 \| \widehat{G}_{K_t} \|_F^2.$$
So, by the second inequality in \eqref{eq:stepsizes}
\begin{equation}\label{eq:interbound1}
    \beta_t \left[\sigma_{min}(D_{\ep}) - \beta_t c_D (\|R\| + c_P \|B\|^2) \right] \|\widehat{G}_{K_t}\|_F^2 \ge (\dfrac{3}{\alpha_t \mu_{\sigma}} + 2) \| \theta_{K_t} - \theta_{K_{t+1}} \|_F^2.
\end{equation}
In addition,
$$\|G_{K_t} - \widehat{G}_{K_t}\|_F^2 = \| (\theta^{22}_{K_t} - \theta^{22}_t) K_t - (\theta^{21}_{K_t} - \theta^{21}_t) \|_F^2 \le c_K^2 \| \theta_t - \theta_{K_t} \|_F^2.$$
So, by the third inequality in \eqref{eq:stepsizes}
\begin{equation}\label{eq:interbound2}
    \frac13 \alpha_t \mu_{\sigma} \| \theta_t - \theta_{K_t} \|_F^2 \ge \beta_t c_D \|G_{K_t} - \widehat{G}_{K_t}\|_F^2.
\end{equation}
Substituting \eqref{eq:interbound1} and \eqref{eq:interbound2} into \eqref{eq:Lyap_diff}, we obtain
\begin{equation*}
\begin{aligned}
     & \quad \EE \left[ \mL_{t+1} ~|~ \mG_t \right] - \mL_t \\
     & \le - \alpha_t \mu_{\sigma} \| \theta_t - \theta_{K_t} \|_F^2 + \frac14 \dfrac{\sigma_{min}(D_{\ep})}{c_3} \beta_t \ve - \beta_t \dfrac{\sigma_{min}(D_{\ep})}{c_3} (J(K_t) - J(K^*)).
\end{aligned}
\end{equation*}
Taking expectation, we obtain
\begin{equation}\label{eq:Lyap_contraction}
    \EE[\mL_{t+1} - \mL_t] \le - \EE \left[ \alpha_t \mu_{\sigma} \| \theta_t - \theta_{K_t} \|_F^2 + \beta_t \dfrac{\sigma_{min}(D_{\ep})}{c_3} (J(K_t) - J(K^*)) \right] + \frac14 \dfrac{\sigma_{min}(D_{\ep})}{c_3} \beta_t \ve.
\end{equation}
Next, we consider three cases. The first case is when $\EE[\|\theta_t - \theta_{K_t}\|^2_F] \ge \frac12 \ve$. In this case, by \eqref{eq:Lyap_contraction} and the first inequality of \eqref{eq:stepsizes}, 
\begin{equation*}
    \EE[\mL_{t+1} - \mL_t] \le - \EE \left[ \frac13 \alpha_t \mu_{\sigma} \| \theta_t - \theta_{K_t} \|_F^2 + \beta_t \dfrac{\sigma_{min}(D_{\ep})}{c_3} (J(K_t) - J(K^*)) \right].
\end{equation*}
The second case is when $\EE[J(K_t) - J(K^*)] \ge \frac12 \ve$. In this case
\begin{equation*}
    \EE[\mL_{t+1} - \mL_t] \le - \EE \left[ \alpha_t \mu_{\sigma} \| \theta_t - \theta_{K_t} \|_F^2 + \frac12 \beta_t \dfrac{\sigma_{min}(D_{\ep})}{c_3} (J(K_t) - J(K^*)) \right].
\end{equation*}
In both the first and the second cases, we have
\begin{equation*}
    \EE[\mL_{t+1} - \mL_t] \le - \EE \left[ \frac13 \alpha_t \mu_{\sigma} \| \theta_t - \theta_{K_t} \|_F^2 + \frac12 \beta_t \dfrac{\sigma_{min}(D_{\ep})}{c_3} (J(K_t) - J(K^*)) \right].
\end{equation*}
Note that $\frac12 \beta_t \dfrac{\sigma_{min}(D_{\ep})}{c_3} \le \frac13 \alpha_t \mu_{\sigma}$, we obtain a contraction rate for the Lyaponov function in both cases:
\begin{equation*}
    \EE[\mL_{t+1} - \mL_t] \le  - \frac12 \beta_t \dfrac{\sigma_{min}(D_{\ep})}{c_3} \EE[\mL_t] =: - \beta_t c_4 \EE[\mL_t]
\end{equation*}
where we remind the reader that $L(\theta_{K^*}, K^*) = 0$. Let us rewrite it into a contraction form
\begin{equation}\label{eq:Lyap_contraction5}
    \EE[\mL_{t+1}] \le  (1 - \beta_t c_4) \EE[\mL_t].
\end{equation}
Next, we consider the third case, when both $\EE[\|\theta_t - \theta_{K_t}\|^2_F] < \frac12 \ve$ and $\EE[J(K_t) - J(K^*)] < \frac12 \ve$. In this case we have $\EE[ \mL_t ] < \ve$. Therefore, by \eqref{eq:Lyap_contraction}, we obtain
\begin{equation*}
\begin{aligned}
    & \quad \EE[\mL_{t+1}] \\
    & \le (1 - \alpha_t \mu_{\sigma}) \EE \left[ \| \theta_t - \theta_{K_t} \|_F^2 \right] + \frac14 \dfrac{\sigma_{min}(D_{\ep})}{c_3} \beta_t \ve + \left(1 - \beta_t \dfrac{\sigma_{min}(D_{\ep})}{c_3} \right) \EE \left[ (J(K_t) - J(K^*))\right]\\
    & < \frac12 \ve + \frac12 \ve \left(\frac12 \dfrac{\sigma_{min}(D_{\ep})}{c_3} \beta_t + 1 - \beta_t \dfrac{\sigma_{min}(D_{\ep})}{c_3} \right) < \ve.
\end{aligned}
\end{equation*}
Therefore, we have shown that under \eqref{eq:error_grad} and \eqref{eq:critic_grad_bound}, the Lyapunov function is decreasing at rate \eqref{eq:Lyap_contraction5} as long as $\EE[\|\theta_t - \theta_{K_t}\|^2_F] \ge \frac12 \ve$ or $\EE[J(K_t) - J(K^*)] \ge \frac12 \ve$, or else, the Lyapunov function will keep being smaller than $\ve$. Since $(1 - \beta_t c_4)^T \mL_0 < \ve$ (recall that $\beta_t$ is constant in $t$), we have $\EE[\mL_T] \le \ve$. Since $\EE[\mL_T]$ is the sum of two non-negative numbers, both of them are less than $\ve$.
\end{proof}

\vskip 0.2in
\bibliography{ref}

\begin{thebibliography}{30}
\providecommand{\natexlab}[1]{#1}
\providecommand{\url}[1]{\texttt{#1}}
\expandafter\ifx\csname urlstyle\endcsname\relax
  \providecommand{\doi}[1]{doi: #1}\else
  \providecommand{\doi}{doi: \begingroup \urlstyle{rm}\Url}\fi

\bibitem[Anderson and Moore(2007)]{anderson2007optimal}
Brian~DO Anderson and John~B Moore.
\newblock \emph{Optimal control: linear quadratic methods}.
\newblock Courier Corporation, 2007.

\bibitem[Arnold and Avez(1968)]{arnold1968ergodic}
Vladimir~Igorevich Arnold and Andr{\'e} Avez.
\newblock \emph{Ergodic problems of classical mechanics}, volume~9.
\newblock Benjamin, 1968.

\bibitem[Bard(2013)]{bard2013practical}
Jonathan~F Bard.
\newblock \emph{Practical bilevel optimization: algorithms and applications},
  volume~30.
\newblock Springer Science \& Business Media, 2013.

\bibitem[Bertsekas(2019)]{bertsekas2019reinforcement}
Dimitri Bertsekas.
\newblock \emph{Reinforcement learning and optimal control}.
\newblock Athena Scientific, 2019.

\bibitem[Bhatnagar et~al.(2009)Bhatnagar, Sutton, Ghavamzadeh, and
  Lee]{bhatnagar2009natural}
Shalabh Bhatnagar, Richard~S Sutton, Mohammad Ghavamzadeh, and Mark Lee.
\newblock Natural actor--critic algorithms.
\newblock \emph{Automatica}, 45\penalty0 (11):\penalty0 2471--2482, 2009.

\bibitem[Bottou(2012)]{bottou2012stochastic}
L{\'e}on Bottou.
\newblock Stochastic gradient descent tricks.
\newblock In \emph{Neural networks: Tricks of the trade}, pages 421--436.
  Springer, 2012.

\bibitem[Bradtke and Barto(1996)]{bradtke1996linear}
Steven~J Bradtke and Andrew~G Barto.
\newblock Linear least-squares algorithms for temporal difference learning.
\newblock \emph{Machine learning}, 22\penalty0 (1):\penalty0 33--57, 1996.

\bibitem[Chen et~al.(2022)Chen, Sun, Xiao, and Yin]{chen2022single}
Tianyi Chen, Yuejiao Sun, Quan Xiao, and Wotao Yin.
\newblock A single-timescale method for stochastic bilevel optimization.
\newblock In \emph{International Conference on Artificial Intelligence and
  Statistics}, pages 2466--2488. PMLR, 2022.

\bibitem[Dean et~al.(2020)Dean, Mania, Matni, Recht, and Tu]{dean2020sample}
Sarah Dean, Horia Mania, Nikolai Matni, Benjamin Recht, and Stephen Tu.
\newblock On the sample complexity of the linear quadratic regulator.
\newblock \emph{Foundations of Computational Mathematics}, 20\penalty0
  (4):\penalty0 633--679, 2020.

\bibitem[Ebrahim et~al.(2010)Ebrahim, Salem, Jain, and
  Badr]{ebrahim2010application}
OS~Ebrahim, MF~Salem, PK~Jain, and MA~Badr.
\newblock Application of linear quadratic regulator theory to the stator
  field-oriented control of induction motors.
\newblock \emph{IET Electric Power Applications}, 4\penalty0 (8):\penalty0
  637--646, 2010.

\bibitem[Fazel et~al.(2018)Fazel, Ge, Kakade, and Mesbahi]{fazel2018global}
Maryam Fazel, Rong Ge, Sham Kakade, and Mehran Mesbahi.
\newblock Global convergence of policy gradient methods for the linear
  quadratic regulator.
\newblock In \emph{International Conference on Machine Learning}, pages
  1467--1476. PMLR, 2018.

\bibitem[Fu et~al.(2020)Fu, Yang, and Wang]{fu2020single}
Zuyue Fu, Zhuoran Yang, and Zhaoran Wang.
\newblock Single-timescale actor-critic provably finds globally optimal policy.
\newblock \emph{arXiv preprint arXiv:2008.00483}, 2020.

\bibitem[Gilks et~al.(1995)Gilks, Richardson, and
  Spiegelhalter]{gilks1995markov}
Walter~R Gilks, Sylvia Richardson, and David Spiegelhalter.
\newblock \emph{Markov chain Monte Carlo in practice}.
\newblock CRC press, 1995.

\bibitem[Hashim(2019)]{hashim2019optimal}
Aamir Hashim.
\newblock Optimal speed control for direct current motors using linear
  quadratic regulator.
\newblock \emph{Journal of Engineering and Computer Science (JECS)},
  14\penalty0 (2):\penalty0 48--56, 2019.

\bibitem[Kakade(2001)]{kakade2001natural}
Sham~M Kakade.
\newblock A natural policy gradient.
\newblock \emph{Advances in neural information processing systems}, 14, 2001.

\bibitem[Kober et~al.(2013)Kober, Bagnell, and Peters]{kober2013reinforcement}
Jens Kober, J~Andrew Bagnell, and Jan Peters.
\newblock Reinforcement learning in robotics: A survey.
\newblock \emph{The International Journal of Robotics Research}, 32\penalty0
  (11):\penalty0 1238--1274, 2013.

\bibitem[Konda and Tsitsiklis(2000)]{konda2000actor}
Vijay~R Konda and John~N Tsitsiklis.
\newblock Actor-critic algorithms.
\newblock In \emph{Advances in neural information processing systems}, pages
  1008--1014, 2000.

\bibitem[Liu et~al.(2020)Liu, Zhang, Basar, and Yin]{liu2020improved}
Yanli Liu, Kaiqing Zhang, Tamer Basar, and Wotao Yin.
\newblock An improved analysis of (variance-reduced) policy gradient and
  natural policy gradient methods.
\newblock In \emph{NeurIPS}, 2020.

\bibitem[Mohammadi et~al.(2021)Mohammadi, Zare, Soltanolkotabi, and
  Jovanovic]{mohammadi2021convergence}
Hesameddin Mohammadi, Armin Zare, Mahdi Soltanolkotabi, and Mihailo~R
  Jovanovic.
\newblock Convergence and sample complexity of gradient methods for the
  model-free linear quadratic regulator problem.
\newblock \emph{IEEE Transactions on Automatic Control}, 2021.

\bibitem[Peters and Schaal(2008)]{peters2008natural}
Jan Peters and Stefan Schaal.
\newblock Natural actor-critic.
\newblock \emph{Neurocomputing}, 71\penalty0 (7-9):\penalty0 1180--1190, 2008.

\bibitem[Silver et~al.(2016)Silver, Huang, Maddison, Guez, Sifre, Van
  Den~Driessche, Schrittwieser, Antonoglou, Panneershelvam, Lanctot,
  et~al.]{silver2016mastering}
David Silver, Aja Huang, Chris~J Maddison, Arthur Guez, Laurent Sifre, George
  Van Den~Driessche, Julian Schrittwieser, Ioannis Antonoglou, Veda
  Panneershelvam, Marc Lanctot, et~al.
\newblock Mastering the game of go with deep neural networks and tree search.
\newblock \emph{nature}, 529\penalty0 (7587):\penalty0 484--489, 2016.

\bibitem[Sinha et~al.(2017)Sinha, Malo, and Deb]{sinha2017review}
Ankur Sinha, Pekka Malo, and Kalyanmoy Deb.
\newblock A review on bilevel optimization: from classical to evolutionary
  approaches and applications.
\newblock \emph{IEEE Transactions on Evolutionary Computation}, 22\penalty0
  (2):\penalty0 276--295, 2017.

\bibitem[Sutton and Barto(2018)]{sutton2018reinforcement}
Richard~S Sutton and Andrew~G Barto.
\newblock \emph{Reinforcement learning: An introduction}.
\newblock MIT press, 2018.

\bibitem[Tu and Recht(2018)]{tu2018least}
Stephen Tu and Benjamin Recht.
\newblock Least-squares temporal difference learning for the linear quadratic
  regulator.
\newblock In \emph{International Conference on Machine Learning}, pages
  5005--5014. PMLR, 2018.

\bibitem[Wiering(2000)]{wiering2000multi}
Marco~A Wiering.
\newblock Multi-agent reinforcement learning for traffic light control.
\newblock In \emph{Machine Learning: Proceedings of the Seventeenth
  International Conference (ICML'2000)}, pages 1151--1158, 2000.

\bibitem[Wu et~al.(2020)Wu, Zhang, Xu, and Gu]{wu2020finite}
Yue Wu, Weitong Zhang, Pan Xu, and Quanquan Gu.
\newblock A finite time analysis of two time-scale actor critic methods.
\newblock \emph{arXiv preprint arXiv:2005.01350}, 2020.

\bibitem[Yang et~al.(2019)Yang, Chen, Hong, and Wang]{yang2019global}
Zhuoran Yang, Yongxin Chen, Mingyi Hong, and Zhaoran Wang.
\newblock On the global convergence of actor-critic: A case for linear
  quadratic regulator with ergodic cost.
\newblock \emph{arXiv preprint arXiv:1907.06246}, 2019.

\bibitem[Yong and Zhou(1999)]{yong1999stochastic}
Jiongmin Yong and Xun~Yu Zhou.
\newblock \emph{Stochastic controls: Hamiltonian systems and HJB equations},
  volume~43.
\newblock Springer Science \& Business Media, 1999.

\bibitem[Zeng et~al.(2021)Zeng, Doan, and Romberg]{zeng2021two}
Sihan Zeng, Thinh~T Doan, and Justin Romberg.
\newblock A two-time-scale stochastic optimization framework with applications
  in control and reinforcement learning.
\newblock \emph{arXiv preprint arXiv:2109.14756}, 2021.

\bibitem[Zhou()]{zhou2021actorcriticPDE-git}
Mo~Zhou.
\newblock Single time-scale actor-critic method to solve the linear quadratic
  regulator with convergence proof.
\newblock \url{https://github.com/MoZhou1995/ActorCriticLQR.git}.

\end{thebibliography}

\end{document}